\documentclass[11pt,a4paper,reqno]{amsart}
\usepackage{mathrsfs,mathtools,hyperref,amssymb}

\allowdisplaybreaks[4]

\theoremstyle{plain}
\newtheorem{thm}{Theorem}
\newtheorem{conj}{Conjecture}
\theoremstyle{remark}
\newtheorem{rem}{Remark}
\DeclareMathOperator{\td}{d}
\DeclareMathOperator{\bell}{B}
\DeclareMathOperator{\te}{e}

\begin{document}

\title[Closed-form formulas, determinantal expressions, ... ] {Closed-form formulas, determinantal expressions, recursive relations, power series, and special values of several functions used in Clark--Ismail's two conjectures}
\author[Y.-F. Li, D. Lim, F. Qi]{Yan-Fang Li$^{1}$, Dongkyu Lim$^{2,*}$, Feng Qi$^{3,4,1,*}$}
\thanks{$^1$School of Mathematics and Informatics, Henan Polytechnic University, Jiaozuo 454010, Henan, China\\
\indent$^2$Department of Mathematics Education, Andong National University, Andong 36729, Republic of Korea\\
\indent$^3$School of Mathematics and Physics, Hulunbuir University, Inner Mongolia 021008, China\\
\indent$^4$Independent researcher, Dallas, TX 75252-8024, USA\\
\indent$^*$Corresponding author
\\ \indent\,\,\,e-mail: dklim@anu.ac.kr (Lim), honest.john.china@gmail.com (Qi)
\\ \indent
  \em \,\,\, This paper will be formally published in Applied and Computational Mathematics \textbf{22} (2023), No.~4}

\begin{abstract}

In the paper, by virtue of the famous formula of Fa\`a di Bruno, with the aid of several identities of partial Bell polynomials, by means of a formula for derivatives of the ratio of two differentiable functions, and with availability of other techniques, the authors establish closed-form formulas in terms of the Bernoulli numbers and the second kind Stirling numbers, present determinantal expressions, derive recursive relations, obtain power series, and compute special values of the function $\frac{v^j}{1-\te^{-v}}$, its derivatives, and related ones used in Clark--Ismail's two conjectures. By these results, the authors also discover a formula for the determinant of a Hessenberg matrix and derive logarithmic convexity of a sequence related to the function and its derivatives.



\end{abstract}

\maketitle

\section{Introduction}

The classical gamma function $\Gamma(v)$, firstly defined by Euler's integral, can also be defined by
\begin{equation*}
\Gamma(v)=\lim_{m\to\infty}\frac{m!m^v}{\prod_{j=0}^m(v+j)}, \quad v\in\mathbb{C}\setminus\{0,-1,-2,\dotsc\}.
\end{equation*}
See~\cite[p.~255, 6.1.2]{abram} and~\cite[Theorem~3.1]{Temme-96-book}.
Its logarithmic derivative
$\psi(v)=[\ln\Gamma(v)]'=\frac{\Gamma'(v)}{\Gamma(v)}$
is called the digamma function.
The gamma function $\Gamma(v)$ and polygamma functions $\psi^{(j)}(v)$ for $j\ge0$ have much extensive applications in many branches such as statistics, probability, number theory, theory of $0$-$1$ matrices, graph theory, combinatorics, physics, engineering, and other mathematical sciences.
\par
Let $I\subseteq\mathbb{R}$ stand for an interval and $\mathbb{N}$ denote the set of all natural numbers (not including zero). If a smooth function $f(v)\in C^\infty(I)$ satisfies $(-1)^j f^{(j)}(v)\ge0$ on $I$ for $j\in\{0\}\cup\mathbb{N}$, then we say that $f(v)$ is of complete monotonicity on $I$. See~\cite[Chapter~XIII]{mpf-1993}, \cite[Chapter~1]{Schilling-Song-Vondracek-2nd}, and~\cite[Chapter~IV]{widder}. A function $f(v)$ is of complete monotonicity on $I=(0,\infty)$ if and only if
\begin{equation}\label{Theorem12b-Laplace}
f(v)=\int_0^\infty \te^{-v s}\td\sigma(s), \quad s\in(0,\infty),
\end{equation}
where $\sigma(s)$ is non-decreasing and the integral~\eqref{Theorem12b-Laplace} converges for $s\in(0,\infty)$. See~\cite[p.~161, Theorem~12b]{widder}.
\par
In~\cite[Lemmas~1 and~3]{alzermonatsh}, Alzer proved that the second derivatives
\begin{equation*}
[v\psi(v)]''=2\psi'(v)+v\psi''(v)
\quad\text{and}\quad
\bigl[v^2\psi(v)\bigr]''=v^2\psi''(v)+4v\psi'(v)+2\psi(v)
\end{equation*}
are both completely monotonic on $(0,\infty)$.
In the paper~\cite{clark-ismail}, among other things, Clark and Ismail proved that
all the derivatives $[v^j\psi(v)]^{(j+1)}$ for $j\in\mathbb{N}$ are completely monotonic on $(0,\infty)$
and that, with the help of the Maple V Release $5$, the $j$th derivatives
\begin{equation*}
\Phi_j^{(j)}(v)=-\bigl[v^j\psi^{(j)}(v)\bigr]^{(j)}
\end{equation*}
for $2\le j\le16$ are completely monotonic on $(0,\infty)$.
In the paper~\cite{clark-ismail}, the proof of the complete monotonicity of $\Phi_j^{(j)}(v)$ relies on the positivity of
\begin{equation*}
f_j(v)=\frac{\td^j}{\td v^j}\biggl(\frac{v^j}{1-\te^{-v}}\biggr)
\quad\text{or}\quad
g_j(v)=\biggl(\frac{1-\te^{-v}}v\biggr)^{j+1}\te^{j v}f_j(v)
\end{equation*}
on $(0,\infty)$ for $2\le j\le16$.
For $j\in\mathbb{N}$, there exists the relation
\begin{equation*}
\Phi_j^{(j)}(v)=\int_0^\infty t^j f_j(t)\te^{-v t}\td t.
\end{equation*}
See~\cite[p.~108]{Alzer-Berg-Koumandos-conj} and~\cite{clark-ismail}.
Thereafter, Clark and Ismail posed in~\cite{clark-ismail} the following two conjectures.

\begin{conj}\label{conjclark1}
The functions $f_j(v)$, or equivalently $g_j(v)$, are positive on $(0,\infty)$ for
all $j>1$.
\end{conj}

\begin{conj}\label{conjclark2}
Let
\begin{equation}\label{conjclark-eq}
g_j(v)=\sum_{k=0}^\infty\gamma(j,k)v^k.
\end{equation}
Then $\gamma(j,k)\ge0$ for all $j>1$ and $k\ge0$.
\end{conj}

Conjecture~\ref{conjclark2} implies Conjecture~\ref{conjclark1}. In~\cite[p.~400]{Musallam-Bustoz-conj}, Conjecture~\ref{conjclark2} was restated as follows. The Maclaurin power series expansion of the function
\begin{equation}\label{frak-g(v)-dfn}
\frak{g}_j(v)=\bigl(\te^v-1\bigr)^{j+1}\biggl(\frac{v^j}{1-\te^{-v}}\biggr)^{(j)}, \quad j\ge0
\end{equation}
has non-negative coefficients. It is clear that $\frak{g}_j(v)=v^{j+1}\te^v g_j(v)$.
\par
Theorem~1.1 in~\cite{Alzer-Berg-Koumandos-conj} reads that there is an integer $j_0$ such that for all $j\ge j_0$ the derivatives $\Phi_j^{(j)}(v)$ are not completely monotonic on $(0,\infty)$. In the proof of~\cite[Theorem~1.1]{Alzer-Berg-Koumandos-conj}, the function
\begin{equation*}
H(v)=\sum_{k=1}^\infty\frac1k\sin\frac{v}k, \quad v\in\mathbb{C},
\end{equation*}
which was investigated by Hardy and Littlewood in 1936, plays a key role, because Conjecture~\ref{conjclark1} is equivalent to the nonnegativity of the function
$\frac12+\frac1{\pi}H\bigl(\frac{v}{2\pi}\bigr)$ for $v>0$.
To verify this equivalence, the authors proved in~\cite[pp.~112--113, Appendix]{Alzer-Berg-Koumandos-conj} that the $f_j(v)$ is nonnegative on $[2\ln2,\infty)$. In~\cite[p.~113]{Alzer-Berg-Koumandos-conj}, the authors wrote the following observation.
\begin{quote}
An examination of the function $f_j(v)$ shows that $f_j(v)$ for $j>1$ starts as an increasing function at $v=0$ with
\begin{equation}\label{f(n)=0-1-der-values}
f_j(0)=\frac{j!}{2},\quad f_j'(0)=\frac{(j+1)!}{12},
\end{equation}
and it oscillates crossing the line $y=j!$ a number of times. It approaches $j!$ from above or below depending on the parity of $j$ for $v\to\infty$. Computer experiments suggest that it crosses $j-1$ times. The oscillation close to $v=0$ becomes very wild as $j$ becomes very large.
\end{quote}
\par
In~\cite[Theorem~2.1]{Musallam-Bustoz-conj}, the positivity of $f_j(v)$ for $j\ge0$ on $(2\ln2,\infty)$ was simply recovered by expressing the function $f_j(v)$ in terms of the Laguerre polynomials.
\par
In~\cite{K.CASTILLO.Ramanujan-2023}, the interval $(2\ln2,\infty)$, where the function $f_j(v)$ for $j\ge0$ is positive, was extended to the longer interval $(\ln2,\infty)$.
\par
In~\cite{Exp-Diff-Ratio-Wei-Guo.tex}, among other things, the following conclusions were obtained.
\begin{enumerate}
\item
For $j\in\{0\}\cup\mathbb{N}$, the function $(-1)^j f_0^{(j)}(v)$ is of complete monotonicity on $(0,\infty)$.
More strongly, the function $f_0(v)$ is of logarithmically complete monotonicity on $(0,\infty)$. For the notion of logarithmically complete monotonicity, please look up in the paper~\cite{compmon2} or in the monograph~\cite{Schilling-Song-Vondracek-2nd}.
\item
For $j\in\{0\}\cup\mathbb{N}$, the functions
\begin{equation*}
1-f_0'(v)-f_0(v) \quad\text{and}\quad (-1)^j\bigl[f_0^{(j+2)}(v)+f_0^{(j+1)}(v)\bigr]
\end{equation*}
are of complete monotonicity on $(0,\infty)$. In particular, the sequence $(-1)^kf_0^{(j)}(v)$ is increasing with respect to $j$, that is, the inequality
\begin{equation*}
(-1)^kf_0^{(j)}(v)<(-1)^{j+1}f_0^{(j+1)}(v)
\end{equation*}
is valid for all $j\in\{0\}\cup\mathbb{N}$ and $v\in(0,\infty)$.
\item
For $j\in\{0\}\cup\mathbb{N}$, the ratio
\begin{equation*}
\mathcal{F}_j(v)=-\frac{f_0^{(j+1)}(v)}{f_0^{(j)}(v)}
\end{equation*}
is decreasing on $(0,\infty)$, with
\begin{equation*}
\lim_{v\to\infty}\mathcal{F}_0(v)=0 \quad\text{and}\quad \lim_{v\to\infty}\mathcal{F}_{j+1}(v)=1.
\end{equation*}
\end{enumerate}
On applications of derivatives of the function $f_0(v)$, please refer to~\cite{AIMS-Math-2019595.tex, RCSM-D-21-00302.tex}.
\par
The function $f_1(v)$ has been investigated in~\cite[Lemma~2.3]{TJI-5(1)(2021)-6.tex} for applying therein. The function $f_2(v)$ has been studied in~\cite[Lemma~2.3]{Alice-tri-half-conj.tex} and~\cite[Lemma~2.1]{TWMS-20657.tex} for applying therein.
\par
The simple function
\begin{equation}\label{F-1(v)-Eq}
F_1(v)=\frac{1-\te^{-v}}v=\int_{0}^{1}\te^{-v u}\td u,
\end{equation}
which is of complete monotonicity on $(-\infty,\infty)$, has been deeply investigated, extensively applied, and expositorily reviewed in~\cite{era-905.tex, Guo-Qi-Filomat-2011-May-12.tex, 1st-Sirling-Number-2012.tex, Qi-Springer-2012-Srivastava.tex, AMSPROC.TEX} and closely related references therein. It is not hard to see that
\begin{equation}\label{F(n)(t)-Int-Eq}
F_j(v)=\frac{1-\te^{-v}}{v^j}=\frac{1}{v^{j-1}}\int_{0}^{1}\te^{-v u}\td u.
\end{equation}
This implies that $F_j(v)$ for $j\ge2$ is of complete monotonicity on $(0,\infty)$, while $F_0(v)$ is a Bernstein function on $(0,\infty)$. For information on the Bernstein functions, please refer to the monograph~\cite{Schilling-Song-Vondracek-2nd} and the papers~\cite{22ICFIDCAA-Filomat.tex, Recipr-Sqrt-Geom-S.tex}.
\par
For $j\in\mathbb{N}$ and $v\in(-\infty,\infty)$, let
\begin{equation}\label{G(n)(t)-dfn}
G_j(v)=\frac{1}{F_j(v)}=
\begin{dcases}
\frac{v^j}{1-\te^{-v}}, & v\ne0;\\
1, & v=0, \quad j=1;\\
0, & v=0, \quad j>1.
\end{dcases}
\end{equation}
It is clear that $f_j(v)=G_j^{(j)}(v)$ for $j\in\mathbb{N}$. In~\cite[Theorem~1.2 and Remark~5.5]{Bessel-ineq-Dgree-CM.tex}, it was proved that, when $1\le j\le 5$, the inequalities
\begin{equation*}
G_1^{(j-1)}(v)\le \frac{I_j\bigl(2\sqrt{v}\,\bigr)}{v^{j/2}}
\quad\text{and}\quad
G_1^{(4)}(v)\le\frac{v+6}{720}
\end{equation*}
are valid on $(0,\infty)$, where
\begin{equation*}
I_\nu(v)= \sum_{j=0}^\infty\frac1{j!\Gamma(\nu+j+1)}\biggl(\frac{v}2\biggr)^{2j+\nu}
\end{equation*}
for $\nu\in\mathbb{R}$ and $v\in\mathbb{C}$ is the first kind modified Bessel function.
\par
The aim of this paper is to further discuss the functions $f_j(v)$, $g_j(v)$, and $\frak{g}_j(v)$ via the functions $F_j(v)$, $G_j(v)$, and their derivatives. Our main results include closed-form formulas in terms of the Bernoulli numbers and the second kind Stirling numbers, determinantal expressions, recursive relations, power series, logarithmic convexity of these functions or their special values, as well as a formula for a Hessenberg determinant in terms of the Bernoulli numbers.

\section{Preliminaries}

In~\cite[Definition~11.2]{Charalambides-book-2002} and~\cite[p.~134, Theorem~A]{Comtet-Combinatorics-74}, partial Bell polynomials, denoted by $\bell_{j,k}(v_1,v_2,\dotsc,v_{j-k+1})$ for $j\ge k\ge0$, are defined by
\begin{equation*}
\bell_{j,k}(v_1,v_2,\dotsc,v_{j-k+1})=\sum_{\substack{1\le i\le j-k+1\\ \ell_i\in\{0\}\cup\mathbb{N}\\ \sum_{i=1}^{j-k+1}i\ell_i=j\\
\sum_{i=1}^{j-k+1}\ell_i=k}}\frac{j!}{\prod_{i=1}^{j-k+1}\ell_i!} \prod_{i=1}^{j-k+1}\biggl(\frac{v_i}{i!}\biggr)^{\ell_i}.
\end{equation*}
The Fa\`a di Bruno formula can be described in terms of $\bell_{j,k}(v_1,v_2,\dotsc,v_{j-k+1})$ by
\begin{equation}\label{Bruno-Bell-Polynomial}
\frac{\td^j}{\td v^j}f\circ h(v)=\sum_{k=0}^nf^{(k)}(h(v)) \bell_{j,k}\bigl(h'(v),h''(v),\dotsc,h^{(j-k+1)}(v)\bigr), \quad j\ge0.
\end{equation}
See~\cite[Theorem~11.4]{Charalambides-book-2002} and~\cite[p.~139, Theorem~C]{Comtet-Combinatorics-74}.
The identities
\begin{align}
\bell_{j,k}\bigl(ab v_1,ab^2v_2,\dotsc,ab^{j-k+1}v_{j-k+1}\bigr) &=a^kb^j\bell_{j,k}(v_1,v_2,\dotsc,v_{j-k+1}), \label{Bell(n-k)}\\
\bell_{j,k}(1,1,\dotsc,1)&=S(j,k) \label{Bell-stirling}
\end{align}
for $j\ge k\ge0$ and $a,b\in\mathbb{C}$ can be found in~\cite[p.~412]{Charalambides-book-2002} and~\cite[p.~135]{Comtet-Combinatorics-74}.
The identity
\begin{equation}\label{B-S-frac-value}
\bell_{j,k}\biggl(\frac12, \frac13,\dotsc,\frac1{j-k+2}\biggr)
=\frac{j!}{(j+k)!}\sum_{\ell=0}^k(-1)^{k-\ell}\binom{j+k}{k-\ell}S(j+\ell,\ell)
\end{equation}
can be found in~\cite[Section~1.8, (1.16)]{Bell-value-elem-funct.tex} and the papers~\cite{MIA-4666.tex, 2Closed-Bern-Polyn2.tex, Special-Bell2Euler.tex, Zhang-Yang-Oxford-Taiwan-12} respectively, where the second kind Stirling numbers $S(j,k)$ can be generated by
\begin{equation*}
\frac{(\te^v-1)^k}{k!}=\sum_{j=k}^\infty S(j,k)\frac{v^j}{j!}.
\end{equation*}
\par
From~\cite[Theorem~2.1]{Eight-Identy-More.tex}, \cite[Theorem~2.1]{exp-derivative-sum-Combined.tex}, and~\cite[Theorem~3.1]{CAM-D-13-01430-Xu-Cen}, we can derive that
\begin{equation}\label{id-gen-new-form1}
\frac{\td^j}{\td v^j}\biggl(\frac1{1-\te^{-v}}\biggr)
=\sum_{\ell=0}^{j}(-1)^{\ell}\ell!S(j+1,\ell+1)\biggl(\frac1{1-\te^{-v}}\biggr)^{\ell+1}, \quad j\ge0.
\end{equation}
\par
Let $g(v)$ and $h(v)\ne0$ be two differentiable functions. Let $U_{(j+1)\times1}(v)$ be an $(j+1)\times1$ matrix whose elements $g_{\ell,1}(v)=g^{(\ell-1)}(v)$ for $1\le \ell\le j+1$, let $V_{(j+1)\times j}(v)$ be an $(j+1)\times j$ matrix whose elements
\begin{equation*}
h_{i,\ell}(v)=
\begin{cases}
\dbinom{i-1}{\ell-1}h^{(i-\ell)}(v), & i-\ell\ge0\\
0, & i-\ell<0
\end{cases}
\end{equation*}
for $1\le i\le j+1$ and $1\le \ell\le j$, and let $|W_{(j+1)\times(j+1)}(v)|$ denote the determinant of the $(j+1)\times(j+1)$ matrix
\begin{equation*}
W_{(j+1)\times(j+1)}(v)=\begin{pmatrix}U_{(j+1)\times1}(v) & V_{(j+1)\times j}(v)\end{pmatrix}.
\end{equation*}
Then the $j$th derivative of the ratio $\frac{g(v)}{h(v)}$ can be computed by
\begin{equation}\label{Sitnik-Bourbaki-reform}
\frac{\td^j}{\td v^j}\biggl[\frac{g(v)}{h(v)}\biggr]
=\frac{(-1)^j}{h^{j+1}(v)}\bigl|W_{(j+1)\times(j+1)}(v)\bigr|.
\end{equation}
This formula is a reformulation of~\cite[p.~40, Exercise~5)]{Bourbaki-Spain-2004}. See also~\cite[Lemma~2.4]{Axioms-1115448.tex}.
\par
The $j$th falling factorial of $v\in\mathbb{C}$ can be defined by
\begin{equation}\label{Fall-Factorial-Dfn-Eq}
\langle v\rangle_j=
\prod_{\ell=0}^{j-1}(v-\ell)=
\begin{cases}
v(v-1)\dotsm(v-j+1), & j\ge1;\\
1,& j=0.
\end{cases}
\end{equation}
\par
Let $H_0=1$ and
\begin{equation*}
H_j=
\begin{vmatrix}
h_{1,1} & h_{1,2} & 0 & \dotsc & 0 & 0\\
h_{2,1} & h_{2,2} & h_{2,3} & \dotsc & 0 & 0\\
h_{3,1} & h_{3,2} & h_{3,3} & \dotsc & 0 & 0\\
\vdots & \vdots & \vdots & \vdots & \vdots & \vdots\\
h_{j-2,1} & h_{j-2,2} & h_{j-2,3} & \dotsc & h_{j-2,j-1} & 0 \\
h_{j-1,1} & h_{j-1,2} & h_{j-1,3} & \dotsc & h_{j-1,j-1} & h_{j-1,j}\\
h_{j,1} & h_{j,2} & h_{j,3} & \dotsc & h_{j,j-1} & h_{j,j}
\end{vmatrix}
\end{equation*}
for $j\in\mathbb{N}$. Then the sequence $H_j$ for $j\ge0$ satisfies $H_1=h_{1,1}$ and
\begin{equation}\label{CollegeMJ-2002-Cahill-Thm}
H_j=\sum_{r=1}^j(-1)^{j-r}h_{j,r} \Biggl(\prod_{\ell=r}^{j-1}h_{\ell,\ell+1}\Biggr) H_{r-1}
\end{equation}
for $j\ge2$, where any empty product is assumed to be $1$. See~\cite[p.~222, Theorem]{CollegeMJ-2002-Cahill} and the proof of~\cite[Theorem~4.1]{Bess-Pow-Polyn-CMES.tex}.

\section{The power series of $f_j(v)$ and more}
In this section, via computing $G_j^{(k)}(0)$ for $k\ge0$ and $j\in\mathbb{N}$, we obtain power series of $f_j(v)$.

\begin{thm}\label{exp-t-n-power-lem}
The following conclusions are valid.
\begin{enumerate}
\item
The derivatives of the completely monotonic function $F_1(v)$ defined on $(-\infty,\infty)$ by~\eqref{F-1(v)-Eq} satisfy
\begin{equation}\label{F(0)-deriv-values}
F_1^{(k)}(0)=\frac{(-1)^k}{k+1}, \quad k\ge0.
\end{equation}
\item
For $k\ge0$ and $j\in\mathbb{N}$, we have\small
\begin{equation}\label{G(n-k-0)-value-eq2}
G_j^{(k)}(0)=
\begin{dcases}
0, & k<j-1;\\
(-1)^{k-j+1}k!(k-j+2)!\sum_{q=0}^{k-j+1}\frac{(-1)^{q}}{q+1} \frac{S(k-j+q+1,q)}{(k-j-q+1)!(k-j+q+1)!}, & k\ge j-1.
\end{dcases}
\end{equation}\normalsize
In particular, we have
\begin{align}\label{G(n)=0-4values}
G_j^{(j-1)}(0)&=(j-1)!, & G_j^{(j)}(0)&=f_j(0)=\frac{j!}{2}, \\
G_j^{(j+1)}(0)&=f_j'(0)=\frac{(j+1)!}{12}, & G_j^{(j+2)}(0)&=0
\label{G(n)=0-4values-later2}
\end{align}
for $j\in\mathbb{N}$.
\item
For $j\in\mathbb{N}$, the power series of $f_j(v)$ is
\begin{equation}\label{Maclaurin-f(n)(t)}
f_j(v)=\sum_{i=0}^{\infty} (-1)^{i+1}(i+1)(i+2)(j+i)! \Biggl[\sum_{q=0}^{i+1}\frac{(-1)^{q}}{q+1} \frac{S(i+q+1,q)}{(i-q+1)!(i+q+1)!}\Biggr] v^i.
\end{equation}
\end{enumerate}
\end{thm}

\begin{proof}
For $k\ge0$, by the integral representation~\eqref{F-1(v)-Eq}, it follows that the derivatives of $F_1(v)$ are
\begin{equation*}
F_1^{(k)}(v)=(-1)^k\int_{0}^{1}u^k\te^{-v u}\td u
\end{equation*}
and satisfy
\begin{equation*}
F_1^{(k)}(0)=(-1)^k\int_{0}^{1}u^k\td u=\frac{(-1)^k}{k+1}.
\end{equation*}
\par
The function $G_j(v)$ defined by~\eqref{G(n)(t)-dfn} can be rewritten as
\begin{equation}\label{G-n(t)-F1(t)-Eq}
G_j(v)=\frac{v^{j-1}}{(1-\te^{-v})/v}=\frac{v^{j-1}}{F_1(v)}.
\end{equation}
By virtue of the Leibnz rule for differentiation and the Fa\`a di Bruno formula~\eqref{Bruno-Bell-Polynomial}, we acquire
\begin{align*}
G_j^{(k)}(v)&=\sum_{\ell=0}^{k}\binom{k}{\ell}\langle j-1\rangle_{\ell}v^{j-\ell-1}\biggl[\frac{1}{F_1(v)}\biggr]^{(k-\ell)}\\
&=\sum_{\ell=0}^{k}\binom{k}{\ell}\langle j-1\rangle_{\ell}v^{j-\ell-1} \sum_{q=0}^{k-\ell}\frac{(-1)^q q!}{[F_1(v)]^{q+1}} \bell_{k-\ell,q}\bigl(F_1'(v),F_1''(v),\dotsc,F_1^{(k-\ell-q+1)}(v)\bigr),
\end{align*}
where $\langle z\rangle_\ell$ is the falling factorial defined by~\eqref{Fall-Factorial-Dfn-Eq}.
Making use of the formula~\eqref{F(0)-deriv-values} and the identities~\eqref{Bell(n-k)} and~\eqref{B-S-frac-value} gives
\begin{equation}
\begin{aligned}\label{bell(0)F(1)-values}
&\quad\lim_{v\to0}\bell_{k-\ell,q}\bigl(F_1'(v),F_1''(v),\dotsc,F_1^{(k-\ell-q+1)}(v)\bigr)\\
&=\bell_{k-\ell,q}\bigl(F_1'(0),F_1''(0),\dotsc,F_1^{(k-\ell-q+1)}(0)\bigr)\\
&=\bell_{k-\ell,q}\biggl(-\frac{1}{2},\frac{1}{3},\dotsc, \frac{(-1)^{k-\ell-q+1}}{k-\ell-q+2}\biggr)\\
&=(-1)^{k-\ell}\bell_{k-\ell,q}\biggl(\frac{1}{2},\frac{1}{3},\dotsc, \frac{1}{k-\ell-q+2}\biggr)\\
&=(-1)^{k-\ell+q}\frac{(k-\ell)!}{(k-\ell+q)!}\sum_{p=0}^q(-1)^{p}\binom{k-\ell+q}{q-p}S(k-\ell+p,p),
\end{aligned}
\end{equation}
where we used the identity~\eqref{Bell(n-k)} and the formula~\eqref{B-S-frac-value}.
Consequently, it follows that
\begin{enumerate}
\item
when $k<j-1$, we have\small
\begin{align*}
\lim_{v\to0}G_j^{(k)}(v)&=\sum_{\ell=0}^{k}\binom{k}{\ell}\lim_{v\to0}\bigl[\langle j-1\rangle_{\ell}v^{j-\ell-1}\bigr] \sum_{m=0}^{k-\ell}(-1)^mm! \bell_{k-\ell,m}\bigl(F_1'(0),F_1''(0),\dotsc, F_1^{(k-\ell-m+1)}(0)\bigr)\\
&=0;
\end{align*}\normalsize
\item
when $k\ge j-1$, we have\small
\begin{align*}
\lim_{v\to0}G_j^{(k)}(v)&=\sum_{\ell=0}^{k}\binom{k}{\ell}\lim_{v\to0}\bigl[\langle j-1\rangle_{\ell}v^{j-\ell-1}\bigr] \sum_{m=0}^{k-\ell}(-1)^mm! \bell_{k-\ell,m}\bigl(F_1'(0),F_1''(0),\dotsc, F_1^{(k-\ell-m+1)}(0)\bigr)\\
&=\binom{k}{j-1}(j-1)!\sum_{m=0}^{k-j+1}(-1)^mm! \bell_{k-j+1,m}\bigl(F_1'(0),F_1''(0),\dotsc, F_1^{(k-j-m+2)}(0)\bigr)\\
&=\frac{k!}{(k-j+1)!}\sum_{m=0}^{k-j+1}(-1)^mm! (-1)^{k-j+1+m}\frac{(k-j+1)!}{(k-j+1+m)!}\\
&\quad\times\sum_{q=0}^m(-1)^{q}\binom{k-j+1+m}{m-q}S(k-j+1+q,q)\\
&=(-1)^{k-j+1}k!\sum_{m=0}^{k-j+1}\frac{m!}{(k-j+1+m)!} \sum_{q=0}^m(-1)^{q}\binom{k-j+1+m}{k-j+1+q}S(k-j+1+q,q)\\
&=(-1)^{k-j+1}k!\sum_{m=0}^{k-j+1} m!\sum_{q=0}^m\frac{(-1)^{q}}{(m-q)!} \frac{S(k-j+q+1,q)}{(k-j+q+1)!}\\
&=(-1)^{k-j+1}k!\sum_{q=0}^{k-j+1}(-1)^{q}\frac{S(k-j+q+1,q)}{(k-j+q+1)!}\sum_{m=q}^{k-j+1} \frac{m!}{(m-q)!}\\
&=(-1)^{k-j+1}k!(k-j+2)!\sum_{q=0}^{k-j+1}\frac{(-1)^{q}}{q+1} \frac{S(k-j+q+1,q)}{(k-j-q+1)!(k-j+q+1)!}.
\end{align*}\normalsize
\end{enumerate}
\par
The power series~\eqref{Maclaurin-f(n)(t)} comes from the relation $G_j^{(j+i)}(v)=f_j^{(i)}(v)$ and the formula~\eqref{G(n-k-0)-value-eq2}.
The proof of Theorem~\ref{exp-t-n-power-lem} is complete.
\end{proof}

\section{A closed-form formula of $f_j(v)$ and more}

In this section, via computing $G_j^{(k)}(v)$ for $j\in\mathbb{N}$ and $k\in\{0\}\cup\mathbb{N}$, we establish a closed-form formula of the function $f_j(v)$.

\begin{thm}\label{First=Formula-f(j)-thm}
For $j\in\mathbb{N}$ and $k\in\{0\}\cup\mathbb{N}$, the derivatives $G_j^{(k)}(v)$ can be computed by
\begin{equation*}
G_j^{(k)}(v)=v^{j-k-1}\sum_{q=0}^{k}\Biggl[\sum_{\ell=q}^{k}(-1)^{\ell-q}\binom{k}{\ell}\langle j\rangle_{k-\ell} (\ell-q)!S(\ell+1,\ell-q+1)\biggl(\frac{v}{1-\te^{-v}}\biggr)^{\ell-q+1}\Biggr]v^q.
\end{equation*}
In particular, we have
\begin{align*}
f_j(v)&=\frac{1}{v}\sum_{q=0}^{j}\Biggl[\sum_{\ell=q}^{j}(-1)^{\ell-q}\binom{j}{\ell}\langle j\rangle_{j-\ell} (\ell-q)!S(\ell+1,\ell-q+1)\biggl(\frac{v}{1-\te^{-v}}\biggr)^{\ell-q+1}\Biggr]v^q,\\
G_j^{(k)}(0)&=
\begin{dcases}
0, & k<j-1;\\
\sum_{q=0}^{j-1}(-1)^qq!\binom{k}{q+k-j+1}\langle j\rangle_{j-q-1} S(q+k-j+2,q+1), & k\ge j-1,
\end{dcases}
\end{align*}
and those two equalities in~\eqref{G(n)=0-4values} and~\eqref{G(n)=0-4values-later2}.
\end{thm}

\begin{proof}
By the Leibniz rule for differentiation of a product and with the aid of the derivative formula~\eqref{id-gen-new-form1}, we have
\begin{align*}
G_j^{(k)}(v)&=\sum_{\ell=0}^{k}\binom{k}{\ell}\biggl(\frac{1}{1-\te^{-v}}\biggr)^{(\ell)} (v^j)^{(k-\ell)}\\
&=\sum_{\ell=0}^{k}\binom{k}{\ell}\langle j\rangle_{k-\ell} \Biggl[\sum_{m=0}^{\ell}(-1)^{m}m!S(\ell+1,m+1)\biggl(\frac1{1-\te^{-v}}\biggr)^{m+1}\Biggr] v^{j-k+\ell}\\
&=\sum_{\ell=0}^{k}\binom{k}{\ell}\langle j\rangle_{k-\ell} \Biggl[\sum_{m=0}^{\ell}(-1)^{m}m!S(\ell+1,m+1)\biggl(\frac{v}{1-\te^{-v}}\biggr)^{m+1}\frac{1}{v^{m+1}}\Biggr] v^{j-k+\ell}\\
&=\sum_{\ell=0}^{k}\binom{k}{\ell}\langle j\rangle_{k-\ell} \Biggl[\sum_{m=0}^{\ell}(-1)^{m}m!S(\ell+1,m+1)\biggl(\frac{v}{1-\te^{-v}}\biggr)^{m+1} v^{\ell-m}\Biggr] v^{j-k-1}\\
&=v^{j-k-1}\sum_{\ell=0}^{k}\binom{k}{\ell}\langle j\rangle_{k-\ell} \Biggl[\sum_{q=0}^{\ell}(-1)^{\ell-q}(\ell-q)!S(\ell+1,\ell-q+1)\biggl(\frac{v}{1-\te^{-v}}\biggr)^{\ell-q+1} v^q\Biggr]\\
&=v^{j-k-1}\sum_{q=0}^{k}\Biggl[\sum_{\ell=q}^{k}(-1)^{\ell-q}\binom{k}{\ell}\langle j\rangle_{k-\ell} (\ell-q)!S(\ell+1,\ell-q+1)\biggl(\frac{v}{1-\te^{-v}}\biggr)^{\ell-q+1}\Biggr]v^q.
\end{align*}
If $0\le k<j-1$, then it is clear that $G_j^{(k)}(0)=0$. If $k=j-1$, then
\begin{align*}
G_j^{(j-1)}(v)&=\sum_{q=0}^{j-1}\Biggl[\sum_{\ell=q}^{j-1}(-1)^{\ell-q}\binom{j-1}{\ell}\langle j\rangle_{j-\ell-1} (\ell-q)!S(\ell+1,\ell-q+1)\biggl(\frac{v}{1-\te^{-v}}\biggr)^{\ell-q+1}\Biggr]v^q\\
&\to\sum_{\ell=0}^{j-1}(-1)^{\ell}\binom{j-1}{\ell}\langle j\rangle_{j-\ell-1}\ell!, \quad v\to0\\
&=(j-1)!\sum_{\ell=0}^{j-1}(-1)^\ell\binom{j}{\ell+1}\\
&=(j-1)!.
\end{align*}
If $k>j-1$, then\small
\begin{align*}
\sum_{\ell=q}^{k}(-1)^{\ell-q}\binom{k}{\ell}\langle j\rangle_{k-\ell} (\ell-q)!S(\ell+1,\ell-q+1)
&=\sum_{\ell=0}^{k-q}(-1)^\ell \ell!\binom{k}{\ell+q}\langle j\rangle_{k-\ell-q}S(\ell+q+1,\ell+1)\\
&=0, \quad 0\le q<k-j+1
\end{align*}\normalsize
and
\begin{align*}
G_j^{(k)}(0)&=\sum_{\ell=k-j+1}^{k}(-1)^{\ell-(k-j+1)}\binom{k}{\ell}\langle j\rangle_{k-\ell} [\ell-(k-j+1)]!S(\ell+1,\ell-(k-j+1)+1)\\
&=\sum_{q=0}^{j-1}(-1)^qq!\binom{k}{q+k-j+1}\langle j\rangle_{j-q-1} S(q+k-j+2,q+1),
\end{align*}
where we used the finity of $G_j^{(k)}(0)$ for $k\ge0$ and $j\in\mathbb{N}$. The finity of $G_j^{(k)}(0)$ for $k\ge0$ and $j\in\mathbb{N}$ can be seen from the formula~\eqref{G(n-k-0)-value-eq2}.
The proof of Theorem~\ref{First=Formula-f(j)-thm} is complete.
\end{proof}

\section{A determinantal expression of $f_j(v)$ and more}

In this section, with the help of the formula~\eqref{Sitnik-Bourbaki-reform} for derivatives of the ratio of two differentiable functions, we simply establish determinantal formulas for the functions $G_j^{(k)}(v)$ and $f_j(v)$.

\begin{thm}\label{f(n-k)-id-Thm}
For $j\in\mathbb{N}$ and $k\ge0$, we have\small
\begin{equation}\label{f(n-k)-identity}
G_j^{(k)}(v)=\frac{(-1)^k}{F_1^{k+1}(v)}
\begin{vmatrix}
\langle j-1\rangle_{0}v^{j-1} & F_1(v) & 0 & \dotsm & 0& 0\\
\langle j-1\rangle_{1}v^{j-2} & F_1'(v) & F_1(v) & \dotsm & 0& 0\\
\langle j-1\rangle_{2}v^{j-3} & F_1''(v) & \binom{2}1F_1'(v) & \dotsm & 0& 0\\
\vdots & \vdots & \vdots & \ddots & \vdots & \vdots\\
\langle j-1\rangle_{k-2}v^{j-k+1} & F_1^{(k-2)}(v) & \binom{k-2}1F_1^{(k-3)}(v) &  \dotsm & F_1(v) & 0\\
\langle j-1\rangle_{k-1}v^{j-k} & F_1^{(k-1)}(v) & \binom{k-1}1F_1^{(k-2)}(v) &  \dotsm & \binom{k-1}{k-2}F_1'(v) & F_1(v)\\
\langle j-1\rangle_{k}v^{j-k-1} & F_1^{(k)}(v) & \binom{k}1F_1^{(k-1)}(v) & \dotsm & \binom{k}{k-2}F_1''(v) & \binom{k}{k-1}F_1'(v)
\end{vmatrix}.
\end{equation}\normalsize
In particular, we obtain\small
\begin{align}\label{f(n-k=n)-identity}
f_j(v)&=\frac{(-1)^j}{F_1^{j+1}(v)}
\begin{vmatrix}
\langle j-1\rangle_{0}v^{j-1} & F_1(v) & 0 & \dotsm & 0& 0\\
\langle j-1\rangle_{1}v^{j-2} & F_1'(v) & F_1(v) & \dotsm & 0& 0\\
\langle j-1\rangle_{2}v^{j-3} & F_1''(v) & \binom{2}1F_1'(v) & \dotsm & 0& 0\\
\vdots & \vdots & \vdots & \ddots & \vdots & \vdots\\
\langle j-1\rangle_{j-2}v & F_1^{(j-2)}(v) & \binom{j-2}1F_1^{(j-3)}(v) &  \dotsm & F_1(v) & 0\\
(j-1)! & F_1^{(j-1)}(v) & \binom{j-1}1F_1^{(j-2)}(v) &  \dotsm & \binom{j-1}{j-2}F_1'(v) & F_1(v)\\
0 & F_1^{(j)}(v) & \binom{j}1F_1^{(j-1)}(v) & \dotsm & \binom{j}{j-2}F_1''(v) & \binom{j}{j-1}F_1'(v)
\end{vmatrix},\\
G_j^{(k)}(0)&=0, \quad 0\le k<j-1,\label{G(n)(k)(0)=0-Eq}
\end{align}\normalsize
and those two equalities in~\eqref{G(n)=0-4values} and~\eqref{G(n)=0-4values-later2}.
\end{thm}

\begin{proof}
From the integral representation~\eqref{F-1(v)-Eq} and the equation~\eqref{G-n(t)-F1(t)-Eq} or from the reciprocal of the equation~\eqref{F(n)(t)-Int-Eq}, we obtain
\begin{equation*}
G_j^{(k)}(v)=\frac{\td^k}{\td v^k}\biggl[\frac{v^{j-1}}{(1-\te^{-v})/v}\biggr]
=\frac{\td^k}{\td v^k}\biggl[\frac{v^{j-1}}{F_1(v)}\biggr].
\end{equation*}
Applying the formula~\eqref{Sitnik-Bourbaki-reform} to the functions $u(v)=v^{j-1}$ and $v(v)=F_1(v)$ directly results in the determinantal expression~\eqref{f(n-k)-identity}.
\par
Setting $k=j$ in~\eqref{f(n-k)-identity} gives~\eqref{f(n-k=n)-identity}.
\par
Taking $v=0$ in~\eqref{f(n-k)-identity} yields~\eqref{G(n)(k)(0)=0-Eq}.
The proof of Theorem~\ref{f(n-k)-id-Thm} is complete.
\end{proof}

\section{A recursive relation of $f_j(v)$ and more}

In this section, by virtue of the recursive relation~\eqref{CollegeMJ-2002-Cahill-Thm}, we simply derive recursive relations for the functions $G_j^{(k)}(v)$ and $f_j(v)$.

\begin{thm}\label{f(j)-Clark-recur-thm}
For $j\in\mathbb{N}$ and $k\ge0$, we have
\begin{equation}\label{f(j)-Clark-recur-eq}
G_j^{(k)}(v)=\frac{1}{F_1(v)}\Biggl[\langle j-1\rangle_{k}v^{j-k-1} -\sum_{r=0}^{k-1}\binom{k}{r}F_1^{(k-r)}(v) G_j^{(r)}(v)\Biggr].
\end{equation}
In particular, we have
\begin{equation}\label{f(j)-k=j-Clark-recur-eq}
f_j(v)=-\frac{1}{F_1(v)}\sum_{r=0}^{j-1}\binom{j}{r}F_1^{(j-r)}(v) G_j^{(r)}(v),
\end{equation}
the first equality in~\eqref{G(n)=0-4values}, and those equalities in~\eqref{G(n)(k)(0)=0-Eq} for $0\le k<j-1$ and $j\in\mathbb{N}$.
\end{thm}

\begin{proof}
In the recursive relation~\eqref{CollegeMJ-2002-Cahill-Thm}, replacing the quantities $k$, $H_{k+1}$, $H_{r-1}$, $h_{i,1}$, and $h_{i,\ell}(v)$ by
\begin{gather*}
k+1, \quad (-1)^kF_1^{k+1}(v)G_j^{(k)}(v), \quad (-1)^{r-2}F_1^{r-1}(v)G_j^{(r-2)}(v), \quad \langle j-1\rangle_{i-1}v^{j-i}
\end{gather*}
for $1\le i\le k+1$, and $\binom{i-1}{\ell-2}F_1^{(i-\ell+1)}(v)$ for $1\le i\le k+1$ and $2\le \ell\le k+1$ in~\eqref{CollegeMJ-2002-Cahill-Thm} respectively arrives at\small
\begin{gather*}
(-1)^kF_1^{k+1}(v)G_j^{(k)}(v)=(-1)^{k}h_{k+1,1} \Biggl(\prod_{\ell=1}^{k}h_{\ell,\ell+1}\Biggr) H_{0}
+\sum_{r=2}^{k+1}(-1)^{k-r+1}h_{k+1,r} \Biggl(\prod_{\ell=r}^{k}h_{\ell,\ell+1}\Biggr) H_{r-1}\\
=(-1)^{k}\langle j-1\rangle_{k}v^{j-k-1} F_1^{k}(v) +\sum_{r=2}^{k+1}(-1)^{k-r+1}\binom{k}{r-2}F_1^{(k-r+2)}(v) F_1^{k-r+1}(v) (-1)^{r-2}F_1^{r-1}(v)G_j^{(r-2)}(v)\\
=(-1)^{k}F_1^{k}(v)\Biggl[\langle j-1\rangle_{k}v^{j-k-1} -\sum_{r=2}^{k+1}\binom{k}{r-2}F_1^{(k-r+2)}(v) G_j^{(r-2)}(v)\Biggr]\\
=(-1)^{k}F_1^{k}(v)\Biggl[\langle j-1\rangle_{k}v^{j-k-1} -\sum_{r=0}^{k-1}\binom{k}{r}F_1^{(k-r)}(v) G_j^{(r)}(v)\Biggr].
\end{gather*}\normalsize
The recursive relation~\eqref{f(j)-Clark-recur-eq} is thus proved.
\par
Taking $k=j$ in~\eqref{f(j)-Clark-recur-eq} and simplifying reduce to~\eqref{f(j)-k=j-Clark-recur-eq}.
\par
Letting $v=0$ in~\eqref{f(j)-Clark-recur-eq} leads to the first equality in~\eqref{G(n)=0-4values} and those equalities in~\eqref{G(n)(k)(0)=0-Eq}.
The proof of Theorem~\ref{f(j)-Clark-recur-thm} is complete.
\end{proof}

\section{A closed-form formula of $\gamma(j,k)$}

In this section, we give a closed-form formula of the coefficients $\gamma(j,k)$ defined in Conjecture~\ref{conjclark2}.

\begin{thm}\label{gamma(n-k)-thm}
For $j\in\mathbb{N}$ and $k\ge0$, the coefficients $\gamma(j,k)$ can be computed by
\begin{equation}
\begin{aligned}\label{gamma(n-k)-Eq}
\gamma(j,k)&=k!\sum_{r=0}^{k}(-1)^{r+1}(r+1)(r+2)(r+j)! \Biggl[\sum_{q=0}^{r+1}\frac{(-1)^{q}}{q+1} \frac{S(r+q+1,q)}{(r-q+1)!(r+q+1)!}\Biggr]\\
&\quad\times\Biggl[\sum_{m=0}^{k-r}(-1)^m\frac{j^{k-r-m}}{(k-r-m)!} \sum_{p=0}^{m}(-1)^p \frac{\langle j+1\rangle_p}{(m+p)!}\sum_{q=0}^p(-1)^{q}\binom{m+p}{p-q}S(m+q,q)\Biggr].
\end{aligned}
\end{equation}
\end{thm}

\begin{proof}
The equation~\eqref{conjclark-eq} means that
\begin{align*}
k!\gamma(j,k)&=\lim_{t\to0}\frac{\td^kg_j(t)}{\td t^k}
=\lim_{t\to0}\frac{\td^k}{\td t^k}\bigl[F_1^{j+1}(t)\te^{nt}f_j(t)\bigr]\\
&=\lim_{t\to0}\sum_{\ell=0}^{k}\binom{k}{\ell}\bigl[F_1^{j+1}(t)\te^{nt}\bigr]^{(\ell)}f_j^{(k-\ell)}(t)\\
&=\lim_{t\to0}\sum_{\ell=0}^{k}\binom{k}{\ell}\Biggl[\sum_{m=0}^{\ell}\binom{\ell}{m} \bigl(F_1^{j+1}(t)\bigr)^{(m)}j^{\ell-m}\te^{nt}\Biggr] f_j^{(k-\ell)}(t)\\
&=\lim_{t\to0}\sum_{\ell=0}^{k}\binom{k}{\ell} G_j^{(j+k-\ell)}(t) \sum_{m=0}^{\ell}\binom{\ell}{m}j^{\ell-m}\te^{nt}\\
&\quad\times\sum_{p=0}^{m}\langle j+1\rangle_p F_1^{j-p+1}(t)\bell_{m,p}\bigl(F_1'(t), F_1''(t), \dotsc, F_1^{(m-p+1)}(t)\bigr)\\
&=\sum_{\ell=0}^{k}\binom{k}{\ell} G_j^{(j+k-\ell)}(0) \sum_{m=0}^{\ell}\binom{\ell}{m}j^{\ell-m}\sum_{p=0}^{m}\langle j+1\rangle_p F_1^{j-p+1}(0)\\
&\quad\times\bell_{m,p}\bigl(F_1'(0), F_1''(0), \dotsc, F_1^{(m-p+1)}(0)\bigr)\\
&=\sum_{\ell=0}^{k}\binom{k}{\ell} (-1)^{k-\ell+1}(j+k-\ell)!(k-\ell+2)!\\
&\quad\times\Biggl[\sum_{q=0}^{k-\ell+1}\frac{(-1)^{q}}{q+1} \frac{S(k-\ell+q+1,q)}{(k-\ell-q+1)!(k-\ell+q+1)!}\Biggr]\\
&\quad\times\Biggl[\sum_{m=0}^{\ell}\binom{\ell}{m}j^{\ell-m} \sum_{p=0}^{m}(-1)^{m+p}\langle j+1\rangle_p \frac{m!}{(m+p)!}\sum_{q=0}^p(-1)^{q}\binom{m+p}{p-q}S(m+q,q)\Biggr]\\
&=(-1)^{k+1}k!\sum_{\ell=0}^{k}(-1)^\ell (k-\ell+2)(k-\ell+1)(k-\ell+j)!\\
&\quad\times\Biggl[\sum_{q=0}^{k-\ell+1}\frac{(-1)^{q}}{q+1} \frac{S(k-\ell+q+1,q)}{(k-\ell-q+1)!(k-\ell+q+1)!}\Biggr]\\
&\quad\times\Biggl[\sum_{m=0}^{\ell}(-1)^m\frac{j^{\ell-m}}{(\ell-m)!} \sum_{p=0}^{m}(-1)^p \frac{\langle j+1\rangle_p}{(m+p)!}\sum_{q=0}^p(-1)^{q}\binom{m+p}{p-q}S(m+q,q)\Biggr]\\
&=k!\sum_{r=0}^{k}(-1)^{r+1}(r+1)(r+2)(r+j)! \Biggl[\sum_{q=0}^{r+1}\frac{(-1)^{q}}{q+1} \frac{S(r+q+1,q)}{(r-q+1)!(r+q+1)!}\Biggr]\\
&\quad\times\Biggl[\sum_{m=0}^{k-r}(-1)^m\frac{j^{k-r-m}}{(k-r-m)!} \sum_{p=0}^{m}(-1)^p \frac{\langle j+1\rangle_p}{(m+p)!}\sum_{q=0}^p(-1)^{q}\binom{m+p}{p-q}S(m+q,q)\Biggr],
\end{align*}
where we used the Leibniz rule for differentiation of a product, the Fa\`a di Bruno formula~\eqref{Bruno-Bell-Polynomial}, the formula~\eqref{G(n-k-0)-value-eq2} in Theorem~\ref{exp-t-n-power-lem}, and the identity expressed by~\eqref{bell(0)F(1)-values},
Theorem~\ref{gamma(n-k)-thm} is thus proved.
\end{proof}

\section{The power series expansion of $\frak{g}_j(v)$}

In this section, we derive power series expansion of the function $\frak{g}_j(v)$ defined by~\eqref{frak-g(v)-dfn}.

\begin{thm}\label{frak-g(v)-Maclayrin-thm}
For $j\in\mathbb{N}$, the power series of the function $\frak{g}_j(v)$ is
\begin{equation}\label{frak-g(v)-Maclayrin-Eq}
\frak{g}_j(v)=(j+1)! \sum_{k=0}^{\infty}C(j,k)\frac{v^k}{k!},
\end{equation}
where\small
\begin{equation}\label{coeficient-frak(g(t))}
C(j,k)=k!\sum_{r=0}^{k}(-1)^{r+1}(r+1)(r+2)(r+j)! \frac{S(k-r,j+1)}{(k-r)!}\sum_{q=0}^{r+1}\frac{(-1)^{q}}{q+1} \frac{S(r+q+1,q)}{(r-q+1)!(r+q+1)!}.
\end{equation}\normalsize
\end{thm}

\begin{proof}
It is easy to see that
\begin{align*}
\lim_{v\to0}\frac{\td^k\frak{g}_j(v)}{\td v^k}&=\lim_{v\to0}\bigl[\bigl(\te^v-1\bigr)^{j+1}G_j^{(j)}(v)\bigr]^{(k)}\\
&=\lim_{v\to0}\sum_{\ell=0}^{k}\binom{k}{\ell}\bigl[\bigl(\te^v-1\bigr)^{j+1}\bigr]^{(\ell)}G_j^{(j+k-\ell)}(v)\\
&=\lim_{v\to0}\sum_{\ell=0}^{k}\binom{k}{\ell}\sum_{p=0}^{\ell}\langle j+1\rangle_{p} \bigl(\te^v-1\bigr)^{j-p+1} \bell_{\ell,p}\bigl(\te^v,\te^v,\dotsc,\te^v\bigr) G_j^{(j+k-\ell)}(v)\\
&=\sum_{\ell=0}^{k}\binom{k}{\ell}\langle j+1\rangle_{j+1} \bell_{\ell,j+1}\bigl(1,1,\dotsc,1\bigr) G_j^{(j+k-\ell)}(0)\\
&=(j+1)!k!\sum_{r=0}^{k}(-1)^{r+1}(r+1)(r+2)(r+j)! \frac{S(k-r,j+1)}{(k-r)!} \\
&\quad\times\sum_{q=0}^{r+1}\frac{(-1)^{q}}{q+1} \frac{S(r+q+1,q)}{(r-q+1)!(r+q+1)!},
\end{align*}
where we used the Leibniz rule for differentiation of a product, the Fa\`a di Bruno formula~\eqref{Bruno-Bell-Polynomial}, the identity~\eqref{Bell-stirling}, and the formula~\eqref{G(n-k-0)-value-eq2} in Theorem~\ref{exp-t-n-power-lem}. The power series~\eqref{frak-g(v)-Maclayrin-Eq} is thus proved. The proof of Theorem~\ref{frak-g(v)-Maclayrin-thm} is complete.
\end{proof}

\section{Computation of a Hessenberg determinant}
We call a matrix $H=(h_{\ell,m})_{n\times n}$ a lower or upper Hessenberg matrix if $h_{\ell,m} = 0$ for all pairs $(\ell,m)$ such that $\ell+1<m$ or $h_{\ell,m}=0$ for all pairs $(\ell,m)$ such that $m+1<\ell$. We call the determinant of a lower or upper Hessenberg matrix the Hessenberg determinant.
\par
In this section, utilizing the formulas~\eqref{F(0)-deriv-values} and~\eqref{G(n-k-0)-value-eq2} in Theorem~\ref{exp-t-n-power-lem}, utilizing the determinantal expression~\eqref{f(n-k)-identity} in Theorem~\ref{f(n-k)-id-Thm}, and employing some basic properties of determinants, we find the following formula for computing a Hessenberg determinant and supply an alternative proof.

\begin{thm}\label{Hessenberg-comput-thm}
For $j\in\{0\}\cup\mathbb{N}$, we have
\begin{equation}\label{Hessenberg-comput-thm-EQ}
\begin{vmatrix}
\frac{1}{2!} & 1 & 0 & \dotsm & 0 & 0 & 0\\
\frac{1}{3!} & \frac{1}{2!} & 1 & \dotsm & 0 & 0 & 0\\
\frac{1}{4!}& \frac{1}{3!} & \frac{1}{2!} & \dotsm & 0 & 0 & 0\\
\vdots & \vdots & \vdots & \ddots & \vdots & \vdots & \vdots\\
\frac{1}{j!} & \frac{1}{(j-1)!} & \frac{1}{(j-2)!} & \dotsm & \frac{1}{2!} &  1 & 0\\
\frac{1}{(j+1)!} & \frac{1}{j!} & \frac{1}{(j-1)!} & \dotsm & \frac{1}{3!} & \frac{1}{2!} & 1 \\
\frac{1}{(j+2)!} & \frac{1}{(j+1)!} & \frac{1}{j!} & \dotsm & \frac{1}{4!} & \frac{1}{3!} & \frac{1}{2!}
\end{vmatrix}
=(-1)^{j+1}\frac{B_{j+1}}{(j+1)!},
\end{equation}
where $B_j$ denotes the Bernoulli numbers which can be generated by
\begin{equation}\label{Bernoullu=No-dfn-Eq}
\frac{v}{\te^v-1}=\sum_{j=0}^\infty B_j\frac{v^j}{j!}
=1-\frac{v}2+\sum_{j=1}^\infty B_{2j}\frac{v^{2j}}{(2j)!}, \quad \vert v\vert<2\pi.
\end{equation}
\end{thm}

\begin{proof}[First proof]
Let $f(v)=1+\sum_{k=1}^\infty a_k v^k$ and $g(v)=1+\sum_{k=1}^\infty b_k v^k$ be two formal power series such that $f(v)g(v)=1$. Then
\begin{equation}\label{f(v)g(v)=1=determ}
b_j=(-1)^j
\begin{vmatrix}
a_1 & 1 & 0 & 0 & \dotsm & 0\\
a_2 & a_1 & 1 & 0 & \dotsm & 0\\
a_3 & a_2 & a_1 & 1 & \dotsm & 0\\
\vdots & \vdots & \vdots & \vdots & \ddots & \vdots\\
a_{j-1} & a_{j-2} & a_{j-3} & a_{j-4} & \dotsm & 1\\
a_{j} & a_{j-1} & a_{j-2} & a_{j-3} & \dotsm & a_1
\end{vmatrix}.
\end{equation}
The formula~\eqref{f(v)g(v)=1=determ} can be found in either~\cite[p.~17, Theorem~1.3]{Henrici-B-1974}, or~\cite[p.~347]{Inselberg-JMAA-1978}, or~\cite[Lemma~2.4]{2Closed-Bern-Polyn2.tex}. Letting $a_k=\frac{1}{(k+1)!}$ yields
\begin{equation*}
f(v)=1+\sum_{k=1}^\infty\frac{v^k}{(k+1)!}
=\frac{\te^v-1}{v}
\quad\text{and}\quad
g(v)=\frac{1}{f(v)}=\frac{v}{\te^v-1}
=1+\sum_{k=1}^\infty B_k\frac{v^k}{k!}.
\end{equation*}
This means that $b_k=\frac{B_k}{k!}$. Making use of these concrete values of $a_k$ and $b_k$ in~\eqref{f(v)g(v)=1=determ} arrives at~\eqref{Hessenberg-comput-thm-EQ}. The first proof of Theorem~\ref{Hessenberg-comput-thm} is complete.
\end{proof}

\begin{proof}[Second proof]
When $k>j-1$ and $v=0$, the determinantal expression~\eqref{f(n-k)-identity} in Theorem~\ref{f(n-k)-id-Thm} can be rearranged as
\begin{align*}
G_j^{(k)}(0)&=(-1)^k
\begin{vmatrix}
0 & F_1(0) & 0 & \dotsm & 0& 0\\
0 & F_1'(0) & F_1(0) & \dotsm & 0& 0\\
0 & F_1''(0) & \binom{2}1F_1'(0) & \dotsm & 0& 0\\
\vdots & \vdots & \vdots & \ddots & \vdots & \vdots\\
0 & F_1^{(j-2)}(0) & \binom{j-2}1F_1^{(j-3)}(0) & \dotsm & 0 & 0\\
\langle j-1\rangle_{j-1} & F_1^{(j-1)}(0) & \binom{j-1}1F_1^{(j-2)}(0) & \dotsm & 0 & 0\\
0 & F_1^{(j)}(0) & \binom{j}1F_1^{(j-1)}(0) & \dotsm & 0 & 0\\
\vdots & \vdots & \vdots & \ddots & \vdots & \vdots\\
0 & F_1^{(k-2)}(0) & \binom{k-2}1F_1^{(k-3)}(0) &  \dotsm & F_1(0) & 0\\
0 & F_1^{(k-1)}(0) & \binom{k-1}1F_1^{(k-2)}(0) &  \dotsm & \binom{k-1}{k-2}F_1'(0) & F_1(0)\\
0 & F_1^{(k)}(0) & \binom{k}1F_1^{(k-1)}(0) & \dotsm & \binom{k}{k-2}F_1''(0) & \binom{k}{k-1}F_1'(0)
\end{vmatrix}\\
&=(-1)^{k+j+1}\langle j-1\rangle_{j-1}
\begin{vmatrix}
A_{(j-1)\times(j-1)} & O_{(j-1)\times(k-j+1)}\\
C_{(k-j+1)\times(j-1)}& B_{(k-j+1)\times(k-j+1)}
\end{vmatrix},
\end{align*}
where, by suitable operations in linear algebra,
\begin{gather*}
\begin{aligned}
\bigl|A_{(j-1)\times(j-1)}\bigr|&=
\begin{vmatrix}
F_1(0) & 0 & 0 & \dotsm & 0 & 0\\
F_1'(0) & F_1(0) & 0 & \dotsm & 0 & 0\\
F_1''(0) & \binom{2}1F_1'(0) & F_1(0) & \dotsm & 0& 0\\
\vdots & \vdots & \vdots & \ddots & \vdots & \vdots\\
F_1^{(j-4)}(0) & \binom{j-4}1F_1^{(j-5)}(0) & \binom{j-4}2F_1^{(j-6)}(0) & \dotsm & 0 & 0\\
F_1^{(j-3)}(0) & \binom{j-3}1F_1^{(j-4)}(0) & \binom{j-3}2F_1^{(j-5)}(0) & \dotsm & F_1(0) & 0 \\
F_1^{(j-2)}(0) & \binom{j-2}1F_1^{(j-3)}(0) & \binom{j-2}2F_1^{(j-4)}(0) & \dotsm & \binom{j-2}{j-3}F_1'(0) & F_1(0)
\end{vmatrix}\\
&=F_1^{j-1}(0)\\
&=1,
\end{aligned}\\
\bigl|B_{(k-j+1)\times(k-j+1)}\bigr|\\
=\begin{vmatrix}
\binom{j}{j-1}F_1'(0) & F_1(0) & 0 & \dotsm & 0 & 0\\
\binom{j+1}{j-1}F_1''(0) & \binom{j+1}{j}F_1'(0) & F_1(0) & \dotsm & 0 & 0\\
\binom{j+2}{j-1}F_1^{(3)}(0) & \binom{j+2}{j}F_1''(0) & \binom{j+2}{j+1}F_1'(0) & \dotsm & 0 & 0\\
\vdots & \vdots & \vdots & \ddots & \vdots & \vdots \\
\binom{k-2}{j-1}F_1^{(k-j-1)}(0) & \binom{k-2}{j}F_1^{(k-j-2)}(0) & \binom{k-2}{j+1}F_1^{(k-j-3)}(0) & \dotsm &  F_1(0) & 0\\
\binom{k-1}{j-1}F_1^{(k-j)}(0) & \binom{k-1}{j}F_1^{(k-j-1)}(0) & \binom{k-1}{j+1}F_1^{(k-j-2)}(0) & \dotsm & \binom{k-1}{k-2}F_1'(0) & F_1(0) \\
\binom{k}{j-1}F_1^{(k-j+1)}(0) & \binom{k}{j}F_1^{(k-j)}(0) & \binom{k}{j+1}F_1^{(k-j-1)}(0) & \dotsm & \binom{k}{k-2}F_1''(0) & \binom{k}{k-1}F_1'(0)
\end{vmatrix}\\
=\begin{vmatrix}
\frac{-j!}{(j-1)!2!} & 1 & 0 & \dotsm & 0 & 0 & 0\\
\frac{(j+1)!}{(j-1)!3!} & \frac{-(j+1)!}{j!2!} & 1 & \dotsm & 0 & 0 & 0\\
\frac{-(j+2)!}{(j-1)!4!}& \frac{(j+2)!}{j!3!} & \frac{-(j+2)!}{(j+1)!2!} & \dotsm & 0 & 0 & 0\\
\vdots & \vdots & \vdots & \ddots & \vdots & \vdots & \vdots\\
\frac{(-1)^{k-j-1}(k-2)!}{(j-1)!(k-j)!} & \frac{(-1)^{k-j-2}(k-2)!}{j!(k-j-1)!} & \frac{(-1)^{k-j-3}(k-2)!}{(j+1)!(k-j-2)!} & \dotsm & \frac{-(k-2)!}{(k-3)!2!} &  1 & 0\\
\frac{(-1)^{k-j}(k-1)!}{(j-1)!(k-j+1)!} & \frac{(-1)^{k-j-1}(k-1)!}{j!(k-j)!} & \frac{(-1)^{k-j-2}(k-1)!}{(j+1)!(k-j-1)!} & \dotsm & \frac{(k-1)!}{(k-3)!3!} & \frac{-(k-1)!}{(k-2)!2!} & 1 \\
\frac{(-1)^{k-j+1}k!}{(j-1)!(k-j+2)!} & \frac{(-1)^{k-j}k!}{j!(k-j+1)!} & \frac{(-1)^{k-j-1}k!}{(j+1)!(k-j)!} & \dotsm & \frac{-k!}{(k-3)!4!} & \frac{k!}{(k-2)!3!} & \frac{-k!}{(k-1)!2!}
\end{vmatrix}\\
=\frac{k!}{(j-1)!}
\begin{vmatrix}
\frac{-1}{2!} & 1 & 0 & \dotsm & 0 & 0 & 0\\
\frac{1}{3!} & \frac{-1}{2!} & 1 & \dotsm & 0 & 0 & 0\\
\frac{-1}{4!}& \frac{1}{3!} & \frac{-1}{2!} & \dotsm & 0 & 0 & 0\\
\vdots & \vdots & \vdots & \ddots & \vdots & \vdots & \vdots\\
\frac{(-1)^{k-j-1}}{((k-j)!} & \frac{(-1)^{k-j-2}}{(k-j-1)!} & \frac{(-1)^{k-j-3}}{(k-j-2)!} & \dotsm & \frac{-1}{2!} &  1 & 0\\
\frac{(-1)^{k-j}}{(k-j+1)!} & \frac{(-1)^{k-j-1}}{(k-j)!} & \frac{(-1)^{k-j-2}}{(k-j-1)!} & \dotsm & \frac{1}{3!} & \frac{-1}{2!} & 1 \\
\frac{(-1)^{k-j+1}}{(k-j+2)!} & \frac{(-1)^{k-j}}{(k-j+1)!} & \frac{(-1)^{k-j-1}}{(k-j)!} & \dotsm & \frac{-1}{4!} & \frac{1}{3!} & \frac{-1}{2!}
\end{vmatrix},
\end{gather*}
and $O_{(j-1)\times(k-j+1)}$ is a matrix whose elements are all zero,
where we used the formula~\eqref{F(0)-deriv-values} in Theorem~\ref{exp-t-n-power-lem}.
Employing the formula
\begin{equation*}
\begin{vmatrix}A & O\\ C & B\end{vmatrix}=|A||B|
\end{equation*}
for square matrices $A$ and $B$, see~\cite[p.~103, eq.~(2.7.6)]{Bernstein=2018-MatrixMath}, we obtain
\begin{align*}
&\quad\frac{k!}{(j-1)!}
\begin{vmatrix}
\frac{-1}{2!} & 1 & 0 & \dotsm & 0 & 0 & 0\\
\frac{1}{3!} & \frac{-1}{2!} & 1 & \dotsm & 0 & 0 & 0\\
\frac{-1}{4!}& \frac{1}{3!} & \frac{-1}{2!} & \dotsm & 0 & 0 & 0\\
\vdots & \vdots & \vdots & \ddots & \vdots & \vdots & \vdots\\
\frac{(-1)^{k-j-1}}{((k-j)!} & \frac{(-1)^{k-j-2}}{(k-j-1)!} & \frac{(-1)^{k-j-3}}{(k-j-2)!} & \dotsm & \frac{-1}{2!} &  1 & 0\\
\frac{(-1)^{k-j}}{(k-j+1)!} & \frac{(-1)^{k-j-1}}{(k-j)!} & \frac{(-1)^{k-j-2}}{(k-j-1)!} & \dotsm & \frac{1}{3!} & \frac{-1}{2!} & 1 \\
\frac{(-1)^{k-j+1}}{(k-j+2)!} & \frac{(-1)^{k-j}}{(k-j+1)!} & \frac{(-1)^{k-j-1}}{(k-j)!} & \dotsm & \frac{-1}{4!} & \frac{1}{3!} & \frac{-1}{2!}
\end{vmatrix}\\
&=(-1)^{k+j+1}\frac{G_j^{(k)}(0)}{\langle j-1\rangle_{j-1}}\\
&=\frac{k!(k-j+2)!}{(j-1)!}\sum_{q=0}^{k-j+1}\frac{(-1)^{q}}{q+1} \frac{S(k-j+q+1,q)}{(k-j-q+1)!(k-j+q+1)!},
\end{align*}
where we used the formula~\eqref{G(n-k-0)-value-eq2} in Theorem~\ref{exp-t-n-power-lem}.
Consequently, it follows that
\begin{gather*}
\begin{vmatrix}
\frac{-1}{2!} & 1 & 0 & \dotsm & 0 & 0 & 0\\
\frac{1}{3!} & \frac{-1}{2!} & 1 & \dotsm & 0 & 0 & 0\\
\frac{-1}{4!}& \frac{1}{3!} & \frac{-1}{2!} & \dotsm & 0 & 0 & 0\\
\vdots & \vdots & \vdots & \ddots & \vdots & \vdots & \vdots\\
\frac{(-1)^{k-j-1}}{((k-j)!} & \frac{(-1)^{k-j-2}}{(k-j-1)!} & \frac{(-1)^{k-j-3}}{(k-j-2)!} & \dotsm & \frac{-1}{2!} &  1 & 0\\
\frac{(-1)^{k-j}}{(k-j+1)!} & \frac{(-1)^{k-j-1}}{(k-j)!} & \frac{(-1)^{k-j-2}}{(k-j-1)!} & \dotsm & \frac{1}{3!} & \frac{-1}{2!} & 1 \\
\frac{(-1)^{k-j+1}}{(k-j+2)!} & \frac{(-1)^{k-j}}{(k-j+1)!} & \frac{(-1)^{k-j-1}}{(k-j)!} & \dotsm & \frac{-1}{4!} & \frac{1}{3!} & \frac{-1}{2!}
\end{vmatrix}\\
=(k-j+2)!\sum_{q=0}^{k-j+1}\frac{(-1)^{q}}{q+1} \frac{S(k-j+q+1,q)}{(k-j-q+1)!(k-j+q+1)!}.
\end{gather*}
Further replacing $k-j$ by $\ell$ and utilizing related properties in linear algebra yield
\begin{equation}
\begin{aligned}\label{Hessenberg-comput-EQ}
&\quad\begin{vmatrix}
\frac{-1}{2!} & 1 & 0 & \dotsm & 0 & 0 & 0\\
\frac{1}{3!} & \frac{-1}{2!} & 1 & \dotsm & 0 & 0 & 0\\
\frac{-1}{4!}& \frac{1}{3!} & \frac{-1}{2!} & \dotsm & 0 & 0 & 0\\
\vdots & \vdots & \vdots & \ddots & \vdots & \vdots & \vdots\\
\frac{(-1)^{\ell-1}}{\ell!} & \frac{(-1)^{\ell-2}}{(\ell-1)!} & \frac{(-1)^{\ell-3}}{(\ell-2)!} & \dotsm & \frac{-1}{2!} &  1 & 0\\
\frac{(-1)^{\ell}}{(\ell+1)!} & \frac{(-1)^{\ell-1}}{\ell!} & \frac{(-1)^{\ell-2}}{(\ell-1)!} & \dotsm & \frac{1}{3!} & \frac{-1}{2!} & 1 \\
\frac{(-1)^{\ell+1}}{(\ell+2)!} & \frac{(-1)^{\ell}}{(\ell+1)!} & \frac{(-1)^{\ell-1}}{\ell!} & \dotsm & \frac{-1}{4!} & \frac{1}{3!} & \frac{-1}{2!}
\end{vmatrix}\\
&=(-1)^{\ell+1}
\begin{vmatrix}
\frac{1}{2!} & 1 & 0 & \dotsm & 0 & 0 & 0\\
\frac{1}{3!} & \frac{1}{2!} & 1 & \dotsm & 0 & 0 & 0\\
\frac{1}{4!}& \frac{1}{3!} & \frac{1}{2!} & \dotsm & 0 & 0 & 0\\
\vdots & \vdots & \vdots & \ddots & \vdots & \vdots & \vdots\\
\frac{1}{\ell!} & \frac{1}{(\ell-1)!} & \frac{1}{(\ell-2)!} & \dotsm & \frac{1}{2!} &  1 & 0\\
\frac{1}{(\ell+1)!} & \frac{1}{\ell!} & \frac{1}{(\ell-1)!} & \dotsm & \frac{1}{3!} & \frac{1}{2!} & 1 \\
\frac{1}{(\ell+2)!} & \frac{1}{(\ell+1)!} & \frac{1}{\ell!} & \dotsm & \frac{1}{4!} & \frac{1}{3!} & \frac{1}{2!}
\end{vmatrix}\\
&=(\ell+2)!\sum_{q=0}^{\ell+1}\frac{(-1)^{q}}{q+1} \frac{S(\ell+q+1,q)}{(\ell-q+1)!(\ell+q+1)!}.
\end{aligned}
\end{equation}
In~\cite[Theorem~1]{Guo-Qi-JANT-Bernoulli.tex}, among other things, it was obtained that
\begin{equation}\label{Bernoulli-Stirling-formula}
B_j=\sum_{i=0}^j(-1)^{i}\frac{\binom{j+1}{i+1}}{\binom{j+i}{i}}S(j+i,i), \quad j\ge0.
\end{equation}
The formula~\eqref{Bernoulli-Stirling-formula} can be rearranged as
\begin{equation}\label{Berlli-Stng-forla}
\begin{aligned}
B_{j+1}&=\sum_{q=0}^{j+1}(-1)^{q}\frac{\binom{j+2}{q+1}}{\binom{j+q+1}{q}}S(j+q+1,q)\\
&=\sum_{q=0}^{j+1}(-1)^{q}\frac{(j+2)!}{(q+1)!(j-q+1)!}\frac{q!(j+1)!}{(j+q+1)!}S(j+q+1,q)\\
&=(j+1)!(j+2)!\sum_{q=0}^{j+1}\frac{(-1)^{q}}{q+1}\frac{S(j+q+1,q)}{(j-q+1)!(j+q+1)!}
\end{aligned}
\end{equation}
for $j\ge0$. Substituting~\eqref{Berlli-Stng-forla} into~\eqref{Hessenberg-comput-EQ} results in the formula~\eqref{Hessenberg-comput-thm-EQ}.
The second proof of Theorem~\ref{Hessenberg-comput-thm} is complete.
\end{proof}

\section{Coefficients in power series in terms of Bernoulli numbers}
In this section, we derive coefficients in power series of the functions $f_j(v)$, $g_j(v)$, and $\frak{g}_j(v)$ in terms of the Bernoulli numbers $B_j$ generated by~\eqref{Bernoullu=No-dfn-Eq}.

\begin{thm}\label{G(n)(0)(k)-Berni-cor}
For $k\ge0$ and $j\in\mathbb{N}$, we have
\begin{equation}\label{G(n)(0)(k)-Berni-cor-Eq}
G_j^{(k)}(0)=
\begin{dcases}
0, & 0\le k<j-1;\\
k!\frac{(-1)^{k-j+1}B_{k-j+1}}{(k-j+1)!}, & k\ge j-1\ge0.
\end{dcases}
\end{equation}
\end{thm}

\begin{proof}[First proof]
From the first proof of Theorem~\ref{Hessenberg-comput-thm}, we obtain
\begin{align*}
G_j^{(k)}(0)&=(-1)^{k+j+1}k!
\begin{vmatrix}
\frac{-1}{2!} & 1 & 0 & \dotsm & 0 & 0 & 0\\
\frac{1}{3!} & \frac{-1}{2!} & 1 & \dotsm & 0 & 0 & 0\\
\frac{-1}{4!}& \frac{1}{3!} & \frac{-1}{2!} & \dotsm & 0 & 0 & 0\\
\vdots & \vdots & \vdots & \ddots & \vdots & \vdots & \vdots\\
\frac{(-1)^{k-j-1}}{((k-j)!} & \frac{(-1)^{k-j-2}}{(k-j-1)!} & \frac{(-1)^{k-j-3}}{(k-j-2)!} & \dotsm & \frac{-1}{2!} &  1 & 0\\
\frac{(-1)^{k-j}}{(k-j+1)!} & \frac{(-1)^{k-j-1}}{(k-j)!} & \frac{(-1)^{k-j-2}}{(k-j-1)!} & \dotsm & \frac{1}{3!} & \frac{-1}{2!} & 1 \\
\frac{(-1)^{k-j+1}}{(k-j+2)!} & \frac{(-1)^{k-j}}{(k-j+1)!} & \frac{(-1)^{k-j-1}}{(k-j)!} & \dotsm & \frac{-1}{4!} & \frac{1}{3!} & \frac{-1}{2!}
\end{vmatrix}\\
&=k!
\begin{vmatrix}
\frac1{2!} & 1 & 0 & \dotsm & 0 & 0 & 0\\
\frac{1}{3!} & \frac1{2!} & 1 & \dotsm & 0 & 0 & 0\\
\frac1{4!}& \frac{1}{3!} & \frac1{2!} & \dotsm & 0 & 0 & 0\\
\vdots & \vdots & \vdots & \ddots & \vdots & \vdots & \vdots\\
\frac1{((k-j)!} & \frac1{(k-j-1)!} & \frac1{(k-j-2)!} & \dotsm & \frac1{2!} &  1 & 0\\
\frac1{(k-j+1)!} & \frac1{(k-j)!} & \frac1{(k-j-1)!} & \dotsm & \frac{1}{3!} & \frac1{2!} & 1 \\
\frac1{(k-j+2)!} & \frac1{(k-j+1)!} & \frac1{(k-j)!} & \dotsm & \frac1{4!} & \frac{1}{3!} & \frac1{2!}
\end{vmatrix}\\
&=k!\frac{(-1)^{k-j+1}B_{k-j+1}}{(k-j+1)!}.
\end{align*}
The first proof of Theorem~\ref{G(n)(0)(k)-Berni-cor} is complete.
\end{proof}

\begin{proof}[Second proof]
For $j\in\mathbb{N}$, the function $G_j(v)$ can be rewritten as
\begin{equation*}
G_j(v)=v^{j-1}\frac{-v}{\te^{-v}-1}
=v^{j-1}\sum_{\ell=0}^{\infty}\frac{B_\ell}{\ell!}(-v)^\ell
=\sum_{\ell=0}^{\infty}(-1)^\ell\frac{B_\ell}{\ell!}v^{j+\ell-1}, \quad |v|<2\pi.
\end{equation*}
Hence, it follows that
\begin{equation*}
G_j^{(k)}(v)=\sum_{\ell=0}^{\infty}(-1)^\ell\frac{B_\ell}{\ell!}\langle j+\ell-1\rangle_kt^{j+\ell-k-1}, \quad |v|<2\pi.
\end{equation*}
Therefore, we arrive at
\begin{equation*}
G_j^{(k)}(0)=
\begin{dcases}
0, & 0\le k<j-1;\\
(-1)^{k-j+1}\frac{B_{k-j+1}}{(k-j+1)!}\langle k\rangle_k, & k\ge j-1\ge0.
\end{dcases}
\end{equation*}
The second proof of Theorem~\ref{G(n)(0)(k)-Berni-cor} is complete.
\end{proof}

\begin{proof}[Third proof]
This follows from combining~\eqref{G(n-k-0)-value-eq2} and~\eqref{Berlli-Stng-forla}. The third proof of Theorem~\ref{G(n)(0)(k)-Berni-cor} is complete.
\end{proof}

\begin{thm}\label{Maclaurin-f(n)(t)-Bern-thm}
For $j\in\mathbb{N}$ and $k\ge0$, we have
\begin{align}\label{Maclaurin-f(n)(t)-Bernoulli}
\frac{f_j(v)}{j!}&=\frac{1}{2}+\sum_{r=1}^{\infty} \binom{j+2r-1}{j} \frac{B_{2r}}{(2r)!} v^{2r-1},\\
\begin{split}\label{gamma(n-k)-Brnn-Eq}
\frac{\gamma(j,k)}{j!k!}&=\sum_{r=0}^{k}\binom{j+r}{j} \frac{(-1)^{r+1}B_{r+1}}{(r+1)!} \Biggl[\sum_{m=0}^{k-r}(-1)^m\frac{j^{k-r-m}}{(k-r-m)!}\\
&\quad\times\sum_{p=0}^{m}(-1)^p \frac{\langle j+1\rangle_p}{(m+p)!}\sum_{q=0}^p(-1)^{q}\binom{m+p}{p-q}S(m+q,q)\Biggr],
\end{split}
\end{align}
and
\begin{equation}\label{coef-Bernoul-frak(g(t))}
\frac{C(j,k)}{j!k!}=\sum_{r=0}^{k}\binom{j+r}{j} \frac{(-1)^{r+1}B_{r+1}}{(r+1)!}\frac{S(k-r,j+1)}{(k-r)!}.
\end{equation}
\end{thm}

\begin{proof}
These follow from combination of~\eqref{Berlli-Stng-forla} with~\eqref{Maclaurin-f(n)(t)} in Theorem~\ref{exp-t-n-power-lem}, \eqref{gamma(n-k)-Eq} in Theorem~\ref{gamma(n-k)-thm}, and~\eqref{coeficient-frak(g(t))} in Theorem~\ref{frak-g(v)-Maclayrin-thm} respectively and from consideration of $B_{2k+1}=0$ for $k\in\mathbb{N}$. The proof of Theorem~\ref{Maclaurin-f(n)(t)-Bern-thm} is complete.
\end{proof}

\section{Logarithmic convexity of the sequence $\bigl|f_j^{(k)}(0)\bigr|$}
In this section, making use of main results in~\cite{RCSM-D-21-00302.tex}, we derive logarithmic convexity of the sequence $\bigl|f_j^{(k)}(0)\bigr|$ and another one.

\begin{thm}\label{f(j)(0)-deriv-thm}
For $j,k\in\mathbb{N}$, we have
\begin{equation*}
f_j(0)=\frac{j!}{2}, \quad f_j^{(2k-1)}(0)=(2k+j-1)! \frac{B_{2k}}{(2k)!}, \quad f_j^{(2k)}(0)=0.
\end{equation*}
\par
For fixed $j\in\mathbb{N}$, the sequence
\begin{equation}\label{f(j)(0)-seq}
\bigl|f_j'(0)\bigr|, \bigl|f_j^{(3)}(0)\bigr|,\dotsc, \bigl|f_j^{(2k-3)}(0)\bigr|, \bigl|f_j^{(2k-1)}(0)\bigr|, \bigl|f_j^{(2k+1)}(0)\bigr|,\dotsc
\end{equation}
is logarithmically convex with respect to $k\in\mathbb{N}$.
\par
For fixed $j\ge6$, the extended sequence
\begin{equation}\label{extended-f(j)(0)-seq}
|f_j(0)|, \bigl|f_j'(0)\bigr|, \bigl|f_j^{(3)}(0)\bigr|,\dotsc, \bigl|f_j^{(2k-3)}(0)\bigr|, \bigl|f_j^{(2k-1)}(0)\bigr|, \bigl|f_j^{(2k+1)}(0)\bigr|,\dotsc
\end{equation}
is logarithmically convex with respect to $k\in\mathbb{N}$.
\end{thm}

\begin{proof}
The power series expansion~\eqref{Maclaurin-f(n)(t)-Bernoulli} can be rewritten as
\begin{equation*}
\frac{f_j(v)}{j!}=\frac{1}{2}+\frac{1}{j!}\sum_{k=1}^{\infty} \biggl[(2k+j-1)! \frac{B_{2k}}{(2k)!}\biggr] \frac{v^{2k-1}}{(2k-1)!}.
\end{equation*}
This means the second equality~\eqref{G(n)=0-4values} and
\begin{equation*}
f_j^{(2k-1)}(0)=(2k+j-1)! \frac{B_{2k}}{(2k)!}, \quad f_j^{(2k)}(0)=0
\end{equation*}
for $k\in\mathbb{N}$.
\par
From~\cite[Theorems~1.1 and~1.2]{RCSM-D-21-00302.tex}, we conclude that, for fixed $\ell\in\{0\}\cup\mathbb{N}$, the sequence
\begin{equation}\label{Bernou-frac-seq-log-conv}
(2k+\ell)!\frac{|B_{2k}|}{(2k)!}
\end{equation}
is logarithmically convex in $k\in\mathbb{N}$. Setting $\ell=j-1$ in~\eqref{Bernou-frac-seq-log-conv} reveals that, for fixed $j\in\mathbb{N}$, the sequence $\bigl|f_j^{(2k-1)}(0)\bigr|$, that is, the sequence in~\eqref{f(j)(0)-seq}, is logarithmically convex with respect to $k\in\mathbb{N}$.
\par
It is easy to verify that the inequality
\begin{equation*}
\bigl|f_j'(0)\bigr|^2=\biggl|(j+1)! \frac{B_{2}}{2!}\biggr|^2=\frac{[(j+1)!]^2}{12^2}
\le |f_j(0)| \bigl|f_j^{(3)}(0)\bigr|
=\frac{j!}{2} (j+3)! \frac{|B_4|}{4!}
=\frac{j!(j+3)!}{1440},
\end{equation*}
which is equivalent to $j^2-5j-4\ge0$, is valid for $j\ge6$. Consequently, the extended sequence in~\eqref{extended-f(j)(0)-seq} is logarithmically convex with respect to $k\in\mathbb{N}$.
The proof of Theorem~\ref{f(j)(0)-deriv-thm} is complete.
\end{proof}

\section{Remarks}

In this section, we give several remarks on our main results in this paper.

\begin{rem}
The first two values in~\eqref{G(n)=0-4values} confirm the examination in~\eqref{f(n)=0-1-der-values}.
\par
When taking $j=1,2,\dotsc,9$ and $k=0,1,2,\dotsc,9$ in the formula~\eqref{G(n-k-0)-value-eq2} or~\eqref{G(n)(0)(k)-Berni-cor-Eq}, we obtain the first few values of $G_j^{(k)}(0)$, which are listed in Table~\ref{G(j)(k)(0)-90}.
\begin{table}[hbtp]
\centering
\caption{The values of $G_j^{(k)}(0)$ for $1\le j\le 9$ and $0\le k\le9$} \label{G(j)(k)(0)-90}
\begin{tabular}{|c|c|c|c|c|c|c|c|c|c|c|}
\hline
 & $k=0$ & $k=1$ & $k=2$ & $k=3$ & $k=4$ & $k=5$ & $k=6$ & $k=7$ & $k=8$ & $k=9$ \\ \hline
$j=1$ & $1$ & $\boldsymbol{\frac{1}{2}}$ & $\frac{1}{6}$ & $0$ & $-\frac{1}{30}$ & $0$ & $\frac{1}{42}$ & $0$ & $-\frac{1}{30}$ & $0$ \\ \hline
$j=2$ & $0$ & $1$ & $\boldsymbol{1}$ & $\frac{1}{2}$ & $0$ & $-\frac{1}{6}$ & $0$ & $\frac{1}{6}$ & $0$ & $-\frac{3}{10}$ \\ \hline
$j=3$ & $0$ & $0$ & $2$ & $\boldsymbol{3}$ & $2$ & $0$ & $-1$ & $0$ & $\frac{4}{3}$ & $0$ \\ \hline
$j=4$ & $0$ & $0$ & $0$ & $6$ & $\boldsymbol{12}$ & $10$ & $0$ & $-7$ & $0$ & $12$ \\ \hline
$j=5$ & $0$ & $0$ & $0$ & $0$ & $24$ & $\boldsymbol{60}$ & $60$ & $0$ & $-56$ & $0$ \\ \hline
$j=6$ & $0$ & $0$ & $0$ & $0$ & $0$ & $120$ & $\boldsymbol{360}$ & $420$ & $0$ & $-504$ \\ \hline
$j=7$ & $0$ & $0$ & $0$ & $0$ & $0$ & $0$ & $720$ & $\boldsymbol{2520}$ & $3360$ & $0$ \\ \hline
$j=8$ & $0$ & $0$ & $0$ & $0$ & $0$ & $0$ & $0$ & $5040$ & $\boldsymbol{20160}$ & $30240$ \\ \hline
$j=9$ & $0$ & $0$ & $0$ & $0$ & $0$ & $0$ & $0$ & $0$ & $40320$ & $\boldsymbol{181440}$ \\ \hline
\end{tabular}
\end{table}
\par
Those numbers with bold fonts in Table~\ref{G(j)(k)(0)-90} are values of $f_j(0)$ for $1\le j\le9$.
\end{rem}

\begin{rem}
From~\eqref{f(n-k)-identity}, letting $v\to0$, it is easy to see that $G_j^{(k)}(0)=0$ for $0\le k<j-1$.
\par
When $k=j-1$, taking $v\to0$ on both sides of~\eqref{f(n-k)-identity} leads to
\begin{align*}
G_j^{(j-1)}(0)&=(-1)^{j-1}
\begin{vmatrix}
0 & F_1(0) & 0 & \dotsm & 0& 0\\
0 & F_1'(0) & F_1(0) & \dotsm & 0& 0\\
0 & F_1''(0) & \binom{2}1F_1'(0) & \dotsm & 0& 0\\
\vdots & \vdots & \vdots & \ddots & \vdots & \vdots\\
0 & F_1^{(j-3)}(0) & \binom{j-3}1F_1^{(j-4)}(0) &  \dotsm & F_1(0) & 0\\
0 & F_1^{(j-2)}(0) & \binom{j-2}1F_1^{(j-3)}(0) &  \dotsm & \binom{j-2}{j-3}F_1'(0) & F_1(0)\\
(j-1)! & F_1^{(j-1)}(0) & \binom{j-1}1F_1^{(j-2)}(0) & \dotsm & \binom{j-1}{j-3}F_1''(0) & \binom{j-1}{j-2}F_1'(0)
\end{vmatrix}\\
&=(j-1)!
\begin{vmatrix}
F_1(0) & 0 & \dotsm & 0& 0\\
F_1'(0) & F_1(0) & \dotsm & 0& 0\\
F_1''(0) & \binom{2}1F_1'(0) & \dotsm & 0& 0\\
\vdots & \vdots & \ddots & \vdots & \vdots\\
F_1^{(j-3)}(0) & \binom{j-3}1F_1^{(j-4)}(0) &  \dotsm & F_1(0) & 0\\
F_1^{(j-2)}(0) & \binom{j-2}1F_1^{(j-3)}(0) &  \dotsm & \binom{j-2}{j-3}F_1'(0) & F_1(0)
\end{vmatrix}\\
&=(j-1)!F_1^{j-1}(0)\\
&=(j-1)!.
\end{align*}
This recovers the first equality in~\eqref{G(n)=0-4values}.
\par
Similarly, we can recover two equalities in~\eqref{f(n)=0-1-der-values} and the left three equalities in~\eqref{G(n)=0-4values} and~\eqref{G(n)=0-4values-later2}.
\end{rem}

\begin{rem}
When taking $j=1,2,\dotsc,9$ and $k=0,1,2,\dotsc,7$ in the formula~\eqref{gamma(n-k)-Eq}, we obtain the first few values of $\gamma(j,k)$, which are listed in Table~\ref{gamma(j-k)-90}.
\begin{table}[hbtp]
\centering
\caption{The values of $\gamma(j,k)$ for $1\le j\le 9$ and $0\le k\le 7$} \label{gamma(j-k)-90}
\begin{tabular}{|c|c|c|c|c|c|c|c|c|}
\hline
  & $k=0$ & $k=1$ & $k=2$ & $k=3$ & $k=4$ & $k=5$ & $k=6$ & $k=7$ \\ \hline
$j=1$ & $\frac{1}{2}$ & $\frac{1}{6}$ & $\frac{1}{24}$ & $\frac{1}{120}$ & $\frac{1}{720}$ & $\frac{1}{5040}$ & $\frac{1}{40320}$ & $\frac{1}{362880}$ \\ \hline
$j=2$ & $1$ & $1$ & $\frac{1}{2}$ & $\frac{13}{72}$ & $\frac{19}{360}$ & $\frac{19}{1440}$ & $\frac{11}{3780}$ & $\frac{349}{604800}$ \\ \hline
$j=3$ & $3$ & $5$ & $4$ & $\frac{13}{6}$ & $\frac{73}{80}$ & $\frac{77}{240}$ & $\frac{11}{112}$ & $\frac{271}{10080}$ \\ \hline
$j=4$ & $12$ & $28$ & $31$ & $\frac{68}{3}$ & $\frac{151}{12}$ & $\frac{103}{18}$ & $\frac{563}{252}$ & $\frac{7789}{10080}$ \\ \hline
$j=5$ & $60$ & $180$ & $255$ & $\frac{707}{3}$ & $\frac{1957}{12}$ & $\frac{365}{4}$ & $\frac{2731}{63}$ & $ \frac{22783}{1260}$ \\ \hline
$j=6$ & $360$ & $1320$ & $2280$ & $2551$ & $\frac{4237}{2}$ & $\frac{5639}{4}$ & $\frac{3159}{4}$ & $\frac{485809}{1260}$ \\ \hline
$j=7$ & $2520$ & $10920$ & $22260$ & $29260$ & $28378$ & $\frac{65734}{3}$ & $\frac{42451}{3}$ & $\frac{712363}{90}$ \\ \hline
$j=8$ & $20160$ & $100800$ & $236880$ & $357840$ & $397152$ & $349224$ & $255543$ & $161154$ \\ \hline
$j=9$ & $181440$ & $1028160$ & $2736720$ & $4672080$ & $5841360$ & $5764752$ & $4715343$ & $3310319$ \\ \hline
\end{tabular}
\end{table}
Can one numerically find a negative value among the coefficients $\gamma(j,k)$ for $j>9$ or $k>7$?
\end{rem}

\begin{rem}
When taking $j=1,2,\dotsc,8$ and $k=0,1,2,\dotsc,9$ in the formula~\eqref{coeficient-frak(g(t))}, we obtain the first few values of $C(j,k)$, which are listed in Table~\ref{C(j-k)-90}.
\begin{table}[hbtp]
\centering
\caption{The values of $C(j,k)$ for $1\le j\le 8$ and $0\le k\le9$} \label{C(j-k)-90}
\begin{tabular}{|c|c|c|c|c|c|c|c|c|c|c|c|}
\hline
 & $k=0$ & $k=1$ & $k=2$ & $k=3$ & $k=4$ & $k=5$ & $k=6$ & $k=7$ & $k=8$ & $k=9$ \\ \hline
$j=1$ & $0$ & $0$ & $\frac{1}{2}$ & $2$ & $\frac{11}{2}$ & $13$ & $\frac{57}{2}$ & $60$ & $\frac{247}{2}$ & $251$ \\ \hline
$j=2$ & $0$ & $0$ & $0$ & $1$ & $8$ & $40$ & $\frac{485}{3}$ & $581$ & $1946$ & $6238$ \\ \hline
$j=3$ & $0$ & $0$ & $0$ & $0$ & $3$ & $40$ & $315$ & $1925$ & $10143$ & $48636$ \\ \hline
$j=4$ & $0$ & $0$ & $0$ & $0$ & $0$ & $12$ & $240$ & $2730$ & $23408$ & $169092$ \\ \hline
$j=5$ & $0$ & $0$ & $0$ & $0$ & $0$ & $0$ & $60$ & $1680$ & $26040$ & $297696$ \\ \hline
$j=6$ & $0$ & $0$ & $0$ & $0$ & $0$ & $0$ & $0$ & $360$ & $13440$ & $272160$ \\ \hline
$j=7$ & $0$ & $0$ & $0$ & $0$ & $0$ & $0$ & $0$ & $0$ & $2520$ & $120960$ \\ \hline
$j=8$ & $0$ & $0$ & $0$ & $0$ & $0$ & $0$ & $0$ & $0$ & $0$ & $20160$ \\ \hline
\end{tabular}
\end{table}
Can one numerically find a negative value among the coefficients $C(j,k)$ for $j>8$ or $k>9$?
\end{rem}

\begin{rem}
The formula~\eqref{Hessenberg-comput-thm-EQ} in Theorem~\ref{Hessenberg-comput-thm} is equivalent to
\begin{equation*}
B_j=j!
\begin{vmatrix}
1 &0&0&\dotsm&0&0&1 \\
\frac1{2!}&1&0&\dotsm&0&0&0 \\
\frac1{3!}&\frac1{2!}&1&\dotsm&0&0&0 \\
\vdots&\vdots&\vdots&\ddots&\vdots&\vdots&\vdots \\
\frac1{(j-1)!}&\frac1{(j-2)!}&\frac1{(j-3)!}&\dotsm&1&0&0 \\
\frac1{j!}&\frac1{(j-1)!}&\frac1{(j-2)!}&\dotsm&\frac1{2!}&1&0 \\
\frac1{(j+1)!}&\frac1{j!}&\frac1{(j-1)!}&\dotsm&\frac1{3!}&\frac1{2!}&0
\end{vmatrix}, \quad j\ge0
\end{equation*}
which was listed in~\cite[Section~21.5]{Korn2-Russian} and~\cite[p.~1]{malenfant2011finite}, can be deduced from a result in~\cite{Booth-Nguyen-2008-09}, and was collected in~\cite{mathematics-131192.tex}. For more information on determinantal expressions of the Bernoulli numbers $B_j$ and the Bernoulli polynomials, please refer to~\cite{2Closed-Bern-Polyn2.tex, mathematics-131192.tex} and closely-related references therein.
\end{rem}

\begin{rem}
Motivated by the formula~\eqref{Hessenberg-comput-thm-EQ} in Theorem~\ref{Hessenberg-comput-thm}, we mention~\cite[Theorem~3.3]{Slovaca-4738.tex} which reads that, for $\theta\in\mathbb{C}$ and $j\in\mathbb{N}$,
\begin{equation*}
\begin{vmatrix}
\binom{\theta}1 & \binom{\theta}0 & 0 & \dotsm &0 & 0 & 0\\
\binom{\theta}2 & \binom{\theta}1 & \binom{\theta}0 & \dotsm & 0 & 0 & 0\\
\binom{\theta}{3} & \binom{\theta}2 & \binom{\theta}1 & \dotsm & 0 & 0 & 0\\
\vdots & \vdots & \vdots & \ddots & \vdots& \vdots & \vdots\\
\binom{\theta}{j-2} & \binom{\theta}{j-3} & \binom{\theta}{j-4} & \dotsm & \binom{\theta}1 & \binom{\theta}0 & 0\\
\binom{\theta}{j-1} & \binom{\theta}{j-2} & \binom{\theta}{j-3} & \dotsm & \binom{\theta}2 & \binom{\theta}1 & \binom{\theta}0\\
\binom{\theta}{j} & \binom{\theta}{j-1} & \binom{\theta}{j-2} & \dotsm & \binom{\theta}{3} & \binom{\theta}2 & \binom{\theta}1
\end{vmatrix}
=\frac{(\theta)_j}{j!},
\end{equation*}
where the quantity
\begin{equation*}
(\theta)_j=\prod_{\ell=0}^{j-1}(\theta+\ell)
=
\begin{cases}
\theta(\theta+1)\dotsm(\theta+j-1), & j\ge1\\
1, & j=0
\end{cases}
\end{equation*}
is called the rising factorial of $\theta\in\mathbb{C}$.
\end{rem}

\begin{rem}
In 26 June 2021, Omran Kouba (Higher Institute for Applied Sciences and Technology, Syria) told the author, Feng Qi, via the ResearchGate that the formula~\eqref{f(v)g(v)=1=determ} can also be found in~\cite[p.~347]{Inselberg-JMAA-1978}. Hereafter, we found that the formula~\eqref{f(v)g(v)=1=determ} also appeared in~\cite[p.~17, Theorem~1.3]{Henrici-B-1974} and~\cite[Section~2]{Rutishauser-ZAMP-1956} and was called the Wronski formula dated back to the French paper~\cite{Wronski-B-1811} in the year 1811.
\par
In~\cite{Whittaker-Edinburgh-1918}, applying the Wronski formula~\eqref{f(v)g(v)=1=determ} from~\cite{Wronski-B-1811}, Whittaker proved that the root of the equation
\begin{equation*}
0=\lambda_0+\lambda_1v+\lambda_2v^2+\lambda_3v^3+\lambda_4v^4+\dotsm,
\end{equation*}
which is the smallest in absolute value, is given by the series
\begin{equation*}
-\frac{\lambda_0}{\lambda_1}-\frac{\lambda_0^2\lambda_2}{\lambda_1\begin{vmatrix}\lambda_1&\lambda_2\\ \lambda_0&\lambda_1\end{vmatrix}}
-\frac{\lambda_0^3\begin{vmatrix}\lambda_2&\lambda_3\\ \lambda_1&\lambda_2\end{vmatrix}}{\begin{vmatrix}\lambda_1&\lambda_2\\ \lambda_0&\lambda_1\end{vmatrix} \begin{vmatrix}\lambda_1&\lambda_2&\lambda_3\\ \lambda_0&\lambda_1&\lambda_2\\ 0&\lambda_0&\lambda_1\end{vmatrix}}
-\frac{\lambda_0^4\begin{vmatrix}\lambda_2&\lambda_3&\lambda_4\\ \lambda_1&\lambda_2&\lambda_3\\ \lambda_0&\lambda_1&\lambda_2\end{vmatrix}} {\begin{vmatrix}\lambda_1&\lambda_2&\lambda_3\\ \lambda_0&\lambda_1&\lambda_2\\ 0&\lambda_0&\lambda_1\end{vmatrix} \begin{vmatrix}\lambda_1&\lambda_2&\lambda_3 &\lambda_4\\ \lambda_0&\lambda_1&\lambda_2&\lambda_3\\ 0&\lambda_0&\lambda_1&\lambda_2\\ 0&0&\lambda_0&\lambda_1\end{vmatrix}}
-\dotsm.
\end{equation*}
\par
Let $P(v)=\lambda_0+\lambda_1v+\lambda_2v^2+\dotsm$ be a formal power series over th field of real numbers such that $\lambda_n>0$ and $\lambda_{n+2}\lambda_n-\lambda_{n+1}^2>0$ for $n=0,1,2,\dotsc$. In~\cite[p.~13, Problem~6]{Henrici-B-1974} and~\cite{Kaluza-MZ-1928}, it was obtained that, if $\frac{1}{P(v)}=\mu_0+\mu_1v+\mu_2v^2+\dotsm$, then $\mu_n<0$ for $n=1,2,\dotsc$. It seems that the Wronski formula~\eqref{f(v)g(v)=1=determ} from~\cite{Wronski-B-1811} was also applied in the paper~\cite{Dahlquist-1959}.
This result and the Wronski formula~\eqref{f(v)g(v)=1=determ} have been applied in~\cite{Posit-Det-Schroder.tex, CDM-73022.tex} to confirm the negativity or the positivity of several Hessenberg determinants whose elements involve the Bernoulli numbers $B_{2j}$ for $j\in\{0\}\cup\mathbb{N}$ and the large Schr\"oder numbers.
\end{rem}

\begin{rem}
In~\cite[Section~3, Theorem, (3.1)]{Alzer-Bernoulli-2000}, it is stated that
\begin{equation}\label{Bernoulli-ineq}
\frac{2(2j)!}{(2\pi)^{2j}} \frac{1}{1-2^{\alpha -2j}} \le(-1)^{j+1} B_{2j} \le \frac{2(2j)!}{(2\pi)^{2j}}\frac{1}{1-2^{\beta -2j}}
\end{equation}
for $j\in\mathbb{N}$, where $\alpha=0$ and $\beta=2+\frac{\ln(1-6/\pi^2)}{\ln2}=0.6491\dotsc$ are the best possible in the sense that they can not be replaced by any bigger and smaller constants in the double inequality~\eqref{Bernoulli-ineq} respectively.
See also~\cite[p.~805, 23.1.15]{abram} and~\cite{CAM-D-18-00067.tex, RCSM-D-21-00302.tex}.
By the double inequality~\eqref{Bernoulli-ineq}, power series~\eqref{Maclaurin-f(n)(t)-Bernoulli} in Theorem~\ref{Maclaurin-f(n)(t)-Bern-thm} can be reformulated as
\begin{equation*}
\frac{f_j(v)}{j!}>\frac{1}{2} +\frac{1}{\pi} \Biggl[\sum_{r=1}^{\infty} \frac{\binom{j+4r-3}{j}}{1-2^{\alpha-4r+2}} \biggl(\frac{v}{2\pi}\biggr)^{4r-3} -\sum_{r=1}^{\infty} \frac{\binom{j+4r-1}{j}}{1-2^{\beta-4r}} \biggl(\frac{v}{2\pi}\biggr)^{4r-1}\Biggr].
\end{equation*}
This may be used to discover a simple and hand-computable lower bound for $f_j(v)$ for $j\in\mathbb{N}$ on the finite interval $(0,2\ln2)$ or on the infinite interval $(0,\infty)$.
\end{rem}

\begin{rem}
Can one read out more useful information from main results, especially from power series~\eqref{Maclaurin-f(n)(t)-Bernoulli} and the quantities~\eqref{gamma(n-k)-Brnn-Eq} and~\eqref{coef-Bernoul-frak(g(t))}, in this paper to find properties of three functions $f_j(v)$, $g_j(v)$, and $\frak{g}_j(v)$?
\end{rem}

\begin{rem}
This paper is a companion of the articles~\cite{AADM-3164.tex, AIMS-Math20210491.tex, Taylor-arccos-v2.tex, ser-sinc-real-pow-Qi-Taylor.tex}.
\end{rem}

\section{Conclusions}
In this paper, by virtue of the Fa\`a di Bruno formula~\eqref{Bruno-Bell-Polynomial}, with the aid of three identities~\eqref{Bell(n-k)}, \eqref{Bell-stirling}, and~\eqref{B-S-frac-value} for partial Bell polynomials $B_{j,k}$, by means of the formula~\eqref{Sitnik-Bourbaki-reform} for higher order derivatives of the ratio of two differentiable functions, and with availability of other techniques, the authors established closed-form formulas in Theorems~\ref{First=Formula-f(j)-thm}, \ref{gamma(n-k)-thm}, and~\ref{Maclaurin-f(n)(t)-Bern-thm} in terms of the Bernoulli numbers $B_{2j}$ and the second kind Stirling numbers $S(j,k)$, presented determinantal expressions in Theorems~\ref{f(n-k)-id-Thm} and~\ref{Hessenberg-comput-thm}, derived recursive relations in Theorem~\ref{f(j)-Clark-recur-thm}, obtained power series expansions in Theorems~\ref{exp-t-n-power-lem} and~\ref{frak-g(v)-Maclayrin-thm}, and computed special values of the function $\frac{v^j}{1-\te^{-v}}$, its derivatives, and related ones used in Clark--Ismail's two conjectures. By these results, the authors also found a formula for a Hessenberg determinant and derived logarithmic convexity in Theorems~\ref{G(n)(0)(k)-Berni-cor} and~\ref{f(j)(0)-deriv-thm}.
\par
Some results reviewed or newly-obtained in this paper have been applied and will be potentially applied to other many fields, areas, subjects, directions, and disciplines. For example, the function defined in~\eqref{F-1(v)-Eq} and its reciprocal, the generating function of the classical Bernoulli numbers, have been applied in~\cite{Guo-Qi-JANT-Bernoulli.tex, Guo-Qi-Filomat-2011-May-12.tex, Alice-tri-half-conj.tex, TJI-5(1)(2021)-6.tex, TWMS-20657.tex, 2Closed-Bern-Polyn2.tex, Recipr-Sqrt-Geom-S.tex, Qi-Springer-2012-Srivastava.tex, AMSPROC.TEX, RCSM-D-21-00302.tex, Exp-Diff-Ratio-Wei-Guo.tex, CAM-D-13-01430-Xu-Cen}, for example.

\subsection*{Funding}
Dongkyu Lim was supported by the National Research Foundation of Korea under Grant No.~NRF-2021R1C1C1010902, Republic of Korea.

\subsection*{Acknowledgments}
The authors appreciate anonymous referees for their careful corrections to, valuable comments on, and helpful suggestions to the original version of this paper.

\end{document}